\documentclass[a4paper, 10pt,reqno]{amsart}
\usepackage[normalem]{ulem}
\usepackage{enumitem}
\usepackage[utf8]{inputenc}
\usepackage[T1]{fontenc}
\usepackage{float}
\usepackage{lmodern}
\usepackage{microtype}
\usepackage{mathtools}
\usepackage{amsthm}
\usepackage{amssymb}
\usepackage{amsfonts}
\usepackage{thmtools}
\usepackage{relsize}
\usepackage[mathscr]{eucal}
\usepackage[linktocpage]{hyperref}
\usepackage[sort,capitalize]{cleveref}
\usepackage[dvipsnames]{xcolor}
\usepackage[msc-links, abbrev, non-sorted-cites]{amsrefs}
\usepackage{xspace}
\makeatletter
\let\amsold@cite\cite
\renewcommand{\cite}[1]{\@ifnextchar*{\@ams@citestar{#1}}{\amsold@cite{#1}\@cite@restorespace}}
\def\@ams@citestar#1*#2{\amsold@cite[#2]{#1}\@cite@restorespace}
\def\@cite@restorespace{\xspace}
\makeatother
\usepackage{todonotes}
\usepackage{caption}
\usepackage{tikz}
\usetikzlibrary{calc, decorations.pathmorphing}

\SetMathAlphabet{\mathsf}{normal}{OT1}{lmss}{m}{n}
\SetMathAlphabet{\mathsf}{bold}{OT1}{lmss}{bx}{n}

\makeatletter
\def\nonumberfootnote{\xdef\@thefnmark{}\@footnotetext}
\makeatother

\definecolor{colorred}{HTML}{B00000}
\definecolor{colorgreen}{HTML}{258300}
\definecolor{colorblue}{HTML}{2e32fa}
\definecolor{coloryellow}{HTML}{cbbb1a}
\hypersetup{colorlinks=true, linkcolor=colorred, citecolor=colorgreen, urlcolor=coloryellow, pdfusetitle=true}

\numberwithin{equation}{section}

\newcommand{\rmO}{{\ensuremath{\mathrm{O}}}}
\newcommand{\rmo}{{\ensuremath{\mathrm{o}}}}
\newcommand{\rmd}{{\ensuremath{\mathrm{d}}}}
\newcommand{\rme}{{\ensuremath{\mathrm{e}}}}

\newcommand{\N}{\boldsymbol{\mathrm{N}}}

\newcommand{\R}{\boldsymbol{\mathrm{R}}}

\renewcommand{\d}{\,\mathrm{d}}

\let\div\undefined

\DeclareMathOperator{\supp}{spt}
\DeclareMathOperator{\div}{div}

\DeclareMathOperator{\tr}{tr}

\theoremstyle{definition}
\newtheorem{bump}{Bump}[section]

\theoremstyle{plain}
\newtheorem{theorem}[bump]{Theorem}
\newtheorem{proposition}[bump]{Proposition}
\newtheorem{definition}[bump]{Definition}
\newtheorem{lemma}[bump]{Lemma}
\newtheorem{corollary}[bump]{Corollary}

\theoremstyle{remark}
\newtheorem{remark}[bump]{Remark}
\newtheorem{example}[bump]{Example}

\newtheorem{convention}[bump]{Convention}

\crefname{theorem}{Theorem}{Theorems}
\crefname{proposition}{Proposition}{Propositions}
\crefname{definition}{Definition}{Definitions}
\crefname{lemma}{Lemma}{Lemmas}
\crefname{corollary}{Corollary}{Corollaries}
\crefname{hypothesis}{Hypothesis}{Hypotheses}
\crefname{remark}{Remark}{Remarks}
\crefname{example}{Example}{Examples}
\crefname{notation}{Notation}{Notations}
\crefname{figure}{Figure}{Figures}

\renewenvironment{example}{\begin{oldexample}}{\hfill $\blacksquare$\end{oldexample}}
\renewenvironment{remark}{\begin{oldremark}}{\hfill $\blacksquare$\end{oldremark}}

\crefformat{section}{{\S}#2#1#3}
\crefformat{subsection}{{\S}#2#1#3}
\crefformat{subsubsection}{{\S}#2#1#3}
\crefformat{appendix}{{\S}#2#1#3}

\crefmultiformat{theorem}{Theorems #2#1#3}{ and #2#1#3}{, #2#1#3}{, and #2#1#3}
\crefmultiformat{proposition}{Propositions #2#1#3}{ and #2#1#3}{, #2#1#3}{, and #2#1#3}
\crefmultiformat{definition}{Definitions #2#1#3}{ and #2#1#3}{, #2#1#3}{, and #2#1#3}
\crefmultiformat{lemma}{Lemmas #2#1#3}{ and #2#1#3}{, #2#1#3}{, and #2#1#3}
\crefmultiformat{corollary}{Corollaries #2#1#3}{ and #2#1#3}{, #2#1#3}{, and #2#1#3}
\crefmultiformat{hypothesis}{Hypotheses #2#1#3}{ and #2#1#3}{, #2#1#3}{, and #2#1#3}
\crefmultiformat{remark}{Remarks #2#1#3}{ and #2#1#3}{, #2#1#3}{, and #2#1#3}
\crefmultiformat{example}{Examples #2#1#3}{ and #2#1#3}{, #2#1#3}{, and #2#1#3}
\crefmultiformat{notation}{Notations #2#1#3}{ and #2#1#3}{, #2#1#3}{, and #2#1#3}

\crefmultiformat{section}{{\S\S}#2#1#3}{ and #2#1#3}{, #2#1#3}{, and #2#1#3}
\crefmultiformat{subsection}{{\S\S}#2#1#3}{ and #2#1#3}{, #2#1#3}{, and #2#1#3}
\crefmultiformat{subsubsection}{{\S\S}#2#1#3}{ and #2#1#3}{, #2#1#3}{, and #2#1#3}

\crefrangeformat{equation}{#3\textcolor{black}{(}#1\textcolor{black}{)}#4 to #5\textcolor{black}{(}#2\textcolor{black}{)}#6}

\newcommand{\mms}{M}
\newcommand{\met}{\mathsf{d}}

\newcommand{\meas}{\mathfrak{m}}

\newcommand{\Leb}{\mathscr{L}}

\newcommand{\vol}{\mathrm{vol}}
\newcommand{\Prob}{\mathscr{P}}

\newcommand{\Id}{\mathrm{Id}}

\newcommand{\pr}{\mathrm{pr}}
\newcommand{\Ric}{\mathrm{Ric}}
\newcommand{\Rm}{\mathrm{Rm}}

\DeclareMathOperator{\Hess}{Hess}

\newcommand{\push}{\sharp}
\def\<{\langle}
\def\>{\rangle}
\def\f{\frac}
\def\epsilon{\varepsilon}

\newcommand{\Len}{\mathrm{Len}}

\newcommand{\mres}{\mathbin{\vrule height 1.6ex depth 0pt width 0.13ex\vrule height 0.13ex depth 0pt width 1.3ex}}

\newcommand{\PT}[1]{{\ensuremath{/\!\!/_{\!{#1}}}}}

\allowdisplaybreaks

\makeatletter
\@namedef{subjclassname@2020}{\textup{2020} Mathematics Subject Classification}
\makeatother

\let\oldtocsection=\tocsection
\let\oldtocsubsection=\tocsubsection
\let\oldtocsubsubsection=\tocsubsubsection

\renewcommand{\tocsection}[2]{\hspace{0em}\oldtocsection{#1}{#2}}
\renewcommand{\tocsubsection}[2]{\hspace{1em}\oldtocsubsection{#1}{#2}}
\renewcommand{\tocsubsubsection}[2]{\hspace{2em}\oldtocsubsubsection{#1}{#2}}

\newcommand{\nocontentsline}[3]{}
\newcommand{\tocless}[2]{\bgroup\let\addcontentsline=\nocontentsline#1{#2}\egroup}
\newcommand{\hEucl}{h_{\textnormal{Eucl}}}

\def\Ric{\operatorname{Ric}}

\begin{document}

\title[Timelike Ollivier--Ricci curvature]{Timelike Ollivier--Ricci curvature}
\author{Mathias Braun}
\author{Xue-Mei Li}
\address{Institute of Mathematics, EPFL, 1015 Lausanne, Switzerland}
\email{\href{mailto:mathias.braun@epfl.ch}{mathias.braun@epfl.ch}}
\email{\href{mailto:xue-mei.li@epfl.ch}{xue-mei.li@epfl.ch}}
\thanks{MB acknowledges financial support by the EPFL through a Bernoulli Instructorship. XML acknowledges support from the EPSRC grant EP/V026100/1, the Swiss National
Science Foundation project MINT (10000849), and supports from NCCR SwissMAP}

\begin{abstract} We introduce a codimension one construction of coarse Ricci
curvature in Lorentzian geometry. Using the $1$-Lorentz--Wasserstein distance $\ell_1$, we compare
probability measures supported on small spacelike hypersurfaces
through nearby events and recover, in a precise asymptotic regime, the ambient Ricci curvature in future-directed unit timelike directions, with a universal dimensional prefactor reflecting this codimension one construction.

The construction echoes the Raychaudhuri equation, which relates
the evolution of spatial volume expansion along timelike geodesics
to Ricci curvature. Our quantitative estimates rely on transport maps with controlled
displacement, constructed through a Moser-type flow.

At leading order, the slice construction is insensitive to smooth
weights. By smearing the spacelike slices into thin timelike tubes and
calibrating the temporal and spatial scales, we recover the
timelike Bakry--\'Emery tensor $\Ric+\Hess V$ associated with
the weighted reference measure $\meas=\rme^{-V}\vol_g$.
\end{abstract}

\maketitle\thispagestyle{empty}
\setcounter{tocdepth}{2}
\tableofcontents

\addtocontents{toc}{\protect\setcounter{tocdepth}{2}}

\section{Introduction}\label{Ch:Intro}

Ricci curvature plays a distinguished role in Lorentzian geometry: notably, Ein\-stein's field equations
\begin{align*}
    \Ric-\frac12 R\,g+\Lambda\, g=\frac{8\pi G}{c^4}\,T
\end{align*}
connect the curvature of spacetime (left-hand side) to the distribution of matter (right-hand side). Here $g$ denotes the  Lorentzian metric, $R$ is the scalar curvature, $\Lambda$ is the cosmological constant, $T$ is the stress-energy tensor encoding matter and energy, $c$ is the speed of light, and Newton's gravitational constant $G$ sets the strength of the coupling between matter and spacetime curvature.

In this article we introduce and study a notion of coarse timelike Ricci curvature and damped / Bakry-Emery Ricci curvature  on a spacetime $(\mms,g)$ of dimension $n+1$, where $n\in\N$, inspired by Ollivier's coarse Ricci curvature for metric spaces \cite{ollivier2007}. We use the convention that the Lorentzian metric has signature $+,-,\dots,-$; in flat spacetime this corresponds to the standard Minkowski metric 
$\smash{\rmd t^2-\rmd x_1^2 - \dots - \rmd x_n^2}$. A Lorentzian manifold $(M,g)$ is time orientable if it admits a globally defined smooth timelike vector field. Equivalently, one can decompose the time cones defined by 
$$
C_x:=\{v\in T_xM:g(v,v)> 0\},
$$
consistently throughout $\mms$ as a disjoint union of future and past cones $\smash{C_x^\pm}$. 
A time orientation is a continuous choice of future cones $(C_x^+)_{x\in M}$. A spacetime is defined as a connected Lorentzian manifold equipped with a time orientation. Locally, after choosing a timelike direction, the tangent bundle decomposes into a timelike line $T$ and its spacelike orthogonal complement $E$: $TM=L\oplus E$. Free-falling massive observers are represented by future-directed timelike geodesics, $\smash{\dot \gamma \in C_{\gamma}^+}$ pointwise, while lightlike geodesics (the trajectories of light rays) lie on its boundary.

 Geometrically, the Ricci curvature $\Ric$ controls the infinitesimal focusing of a congruence of timelike geodesics: \emph{positive} timelike Ricci curvature makes neighbouring free-falling observers converge, and the Raychaudhuri equation 
converts this focusing, under the strong energy condition $\Ric\geq 0$ in all timelike directions, into the classical singularity theorems of Penrose and Hawking \cite{penrose1965,hawking1967,hawking-penrose1970}; cf.~the monographs of Hawking--Ellis \cite{hawking-ellis1973} and Wald \cite{wald1984}. In this sense the \emph{timelike} Ricci curvature is the spacetime counterpart of the central object of Riemannian comparison geometry, and lower bounds on it encode the energy conditions of mathematical relativity. The aim of this paper is to recover this quantity, at a single event and in a single timelike direction, from a purely transport-theoretic comparison of probability measures, and to do so by a mechanism that survives the passage to a discrete, order-theoretic substitute for the spacetime.

 This congruence of curves point of view is the guiding principle in this paper, we return to this shortly.

\subsection{Coarse Ricci curvature on Riemannian manifolds} On Riemannian manifolds and, more generally, on metric measure spaces, optimal transport has become the canonical tool for an intrinsic, synthetic notion of Ricci curvature. Following the seminal works of Sturm \cite{sturm2006-i,sturm2006-ii} and Lott--Villani \cite{lott-villani2009}, lower Ricci bounds are characterized by the convexity of an entropy functional along $W_2$-geodesics of probability measures, with no reference to a smooth structure; for an overview, we refer to Villani's monograph  \cite{villani2009}. A complementary, more elementary device is Ollivier's \emph{coarse Ricci curvature} \cite{ollivier2007,ollivier2009}. (Similar expansions for the $1$-Wasserstein distance were given by von Renesse--Sturm \cite{von-renesse-sturm2005}; for an alternative approach to Ricci curvature for Markov chains, see Maas \cite{maas2011}.) It is defined on any metric space $(\mms,\met)$ carrying a family of probability measures $\meas\colon\mms\to\Prob(\mms)$ (sometimes thought of as a random walk) by comparing, for two distinct points $x,y\in\mms$, the $1$-Wasserstein distance of $\meas_x$ and $\meas_y$ with the distance of $x$ and $y$, namely
\begin{align}
\kappa(x,y) := 1 - \frac{W_1(\meas_x,\meas_y)}{\met(x,y)}.
\end{align}
When $\meas_x$ is the normalized uniform or weighted measure on a small geodesic ball of radius $\varepsilon$ around $x$ on a Riemannian manifold, $\kappa$ recovers the Ricci curvature in the limit $\varepsilon\to 0$, up to an explicit multiplicative constant \cite{ollivier2009}.  
 For our purposes, the decisive technical input is the subsequent work of Arnaudon--Li--Petko \cite{arnaudon-li-petko2025}: it provides a rigorous proof of the required expansion, both in the unweighted and weighted settings, and, in the latter case, constructs the approximate optimal transport map adapted to the weight. This makes the second-order coefficient explicit in a uniform, quantitative form, robust enough to support extensions beyond the smooth manifold framework.
Crucially for the present program, these ideas also survive a \emph{discretization} of the base space. Van der Hoorn--Lippner--Trugenberger--Krioukov \cite{vanderhoorn-lippner-trugenberger-krioukov2023} showed that, on a random geometric graph sampled from a Riemannian manifold, the coarse Ricci curvature of the graph converges, as the sampling density grows, to the Ricci curvature of the underlying manifold. Thus, in positive signature, the smooth curvature tensor is already encoded in the coarse comparison of measures attached to finitely many points and their mutual distances, without reference to the ambient smooth structure.

\subsection{Reconstruction of spacetimes and causal set theory}
 The Lorentzian counterpart of this discretization problem is sharply posed by \emph{causal set theory}, the approach to quantum gravity of Bombelli--Lee--Meyer--Sorkin \cite{bombelli-lee-meyer-sorkin1987}, in which a spacetime is replaced by a locally finite partially ordered set, the order relation encoding causality. The program rests on the theorem of Hawking--King--McCarthy \cite{hawking-king-mccarthy1976} and Malament \cite{malament1977}: a bijection between suitable spacetimes preserving the causal order is already a conformal isometry, so that causal order plus a volume element determines the geometry. In practice, a causal set is obtained from a spacetime by a \emph{sprinkling} of points distributed in accordance with the volume measure, and the central tenet of the theory is that the discrete order, together with the cardinality serving as volume, retains the macroscopic geometry; cf.~Rideout--Sorkin \cite{rideout-sorkin2000} and Sorkin \cite{sorkin1997}. The problem of \emph{reconstructing} a spacetime, or its geometric invariants, from such order-theoretic data is therefore the structural heart of the subject, see the review of Surya \cite{surya2019} for an overview. Rigorous reconstruction results in the smooth and synthetic Lorentzian setting have begun to appear in the work of Braun--S\"amann \cite{braun-samann+} and Braun \cite{braun2025+-order}. Probabilistic reconstructions of certain quantities associated to spacetime have been achieved e.g.~for dimension and volume (Myrheim \cite{myrheim1978}), time separation function (Brightwell--Gregory \cite{brightwell-gregory1991}), or scalar curvature (Benincasa--Dowker \cite{benincasa-dowker2010}).

It is thus natural to ask whether \emph{timelike Ricci curvature}, in particular, can be detected from such order-theoretic data by an Ollivier-type comparison of measures on causal diamonds or spacelike slices. Numerical evidence that this is so was recently provided by Barton--Borza--R\"ohrig \cite{barton-borza-rohrig2026+}, whose experiments recover the expected curvature for the constant-curvature spacetimes of Minkowski, de Sitter, and anti-de Sitter. The purpose of the present paper is to establish, in rigorous form, the continuum mechanism underlying such a detection: a precise asymptotic identity --- in both the unweighted and the weighted setting --- that recovers the timelike Ricci tensor from the $1$-Lorentz--Wasserstein distance between probability measures in the sense of Eckstein--Miller \cite{eckstein-miller2017} supported near two events.

This reconstruction perspective connects our work with manifold learning and geometric inverse problems, where one seeks to recover geometric information from measurements. Related reconstruction results in Lorentzian geometry appear in \cite{lassas-oksanen-yang2016}, where time-separation measurements determine the local metric jet under suitable geometric assumptions, and in \cite{kurylev-lassas-oksanen-uhlmann2022}, where active wave measurements determine the conformal spacetime structure and, in the vacuum case, the metric itself.

\subsection{Intrinsic vs.~extrinsic curvature}
We transfer the coarse construction of Ollivier \cite{ollivier2007} and Arnaudon--Li--Petko \cite{arnaudon-li-petko2025} to Lorentzian signature, the role of the metric being played by the time separation function $l$ and that of the $1$-Wasserstein distance by its Lorentzian analog $\ell_1$, cf.~\cref{Def:LW}. 

When one works with a timelike geodesic congruence issued from a spacelike hypersurface, it becomes natural to keep track not only of the intrinsic geometry of the spacetime, but also of the embedded geometry of the initial hypersurface. In this sense, the present article is close in spirit to Arnaudon--Li--Petko \cite{arnaudon-li-petko2025-ex}, where a coarse second fundamental form was introduced. Before explaining this connection in the next section, we describe the approach used here.

Two obstructions must be overcome. First, in a spacetime there are no balls that are simultaneously small and uniformly comparable in all directions; the natural object replacing the geodesic ball is anisotropic. Second, and more seriously, the time separation function is not a distance, and there is no transport between two measures unless the entire support of the first marginal lies in the chronological past of that of the second. The functional $\ell_1$ is a \emph{supremum} rather than an infimum, so its asymptotics behave oppositely to the Riemannian case: the second-order coefficient carries the opposite sign.

We resolve the first obstruction by supporting the reference measures on small \emph{spacelike} hypersurface patches transverse to a fixed future-directed timelike geodesic with initial velocity $v\in T\mms$, a given timelike vector with unit speed. This is the simplest and most transparent choice: the measures are codimension-one, the transport between two transverse slices is governed at leading order by the geodesic flow alone, and the spatial curvature of the slice drops out of the relevant order.

The formulation of coarse Ricci curvature can be motivated by the so-called damped parallel transport equation along a geodesic, cf.~Arnaudon--Li--Petko \cite{arnaudon-li-petko2025}*{§1}. In contrast, our approach echoes the Raychaudhuri equation, where the Ricci term is recovered from the rate of change of the spatial volume carried by a small family of neighboring geodesics, together with the quadratic contributions of expansion, shear, and vorticity. More precisely, for a timelike geodesic congruence with unit tangent field $U\in\mathscr{X}(\mms)$, the Raychaudhuri equation can be rearranged along each induced geodesic $\gamma\colon I\to\mms$ as follows for every $t\in I$:
\begin{equation}\label{Raychaudhuri}
\operatorname{Ric}_{\gamma(t)}(U,U)
=-\frac{\rmd}{\rmd t}\theta_{\gamma(t)}
-\frac1n\,\theta_{\gamma(t)}^2 
-g(\sigma_{\gamma(t)},\sigma_{\gamma(t)}) + g(\omega_{\gamma(t)},\omega_{\gamma(t)}).
\end{equation}
where $n=\dim(M)-1$. Thus $\operatorname{Ric}(U,U)$ appears as a curvature term detected by the deformation of the congruence. The spatial covariant derivative $\nabla^\perp U$ measures the infinitesimal deformation of neighboring geodesics. Its trace $\smash{\theta:=\operatorname{tr}(\nabla^\perp U)}$ --- where $\tr$ denotes the trace with respect to the restriction of $g$ to $\smash{U^\perp}$ --- is the expansion scalar: it measures the infinitesimal logarithmic rate of change of the spatial volume element transverse to the congruence. Its symmetric trace-free part is the shear $\sigma$ and its antisymmetric part is the vorticity $\omega$, respectively.

\subsection{Main results}
The first main result of this paper, to be found in \cref{Th:MainInformal}, is the asymptotic expansion of the transport distance between test measures in which the timelike Ricci curvature enters. Throughout, $\ell_1$ denotes the $1$-Lorentz--Wasserstein distance of Eckstein--Miller \cite{eckstein-miller2017}. Throughout the article we assume a smallness convention, \cref{Re:SmallEnough},  and the following.

\begin{convention}[Spacetime and reference geodesic]\label{assumption1} Let 
 $(\mms,g)$ designate a globally hyperbolic spacetime of dimension $n+1$, where $n\in\N$, and $V\in C^\infty(\mms)$. 
Let $\gamma\colon I\to\mms$ denote a timelike geodesic with unit speed velocity $v$ at zero; $\dot\gamma^\perp$ denotes the $n$-dimensional bundle along $\gamma$ orthogonal to $\dot \gamma$, and
\begin{align*}
    B_\varepsilon^{h_t}(0)=\{w\in \dot\gamma(t)^\perp: g(w,w)=-\varepsilon^2\}
\end{align*}
denotes the spacelike geodesic ball of radius $\varepsilon>0$ transverse to $\dot\gamma(t)$.
\end{convention}

Consider a probability measure on $\smash{B_\varepsilon^{h_t}(0)}$,
\begin{align}\label{Eq:XiY}
    \xi_{\gamma(t)}^\varepsilon := \Big[\!\int_{B_\varepsilon^{h_t}(0)} \rme^{-V\circ\,\exp_{\gamma(t)}}\d\vol_{h_t}\Big]^{-1}\,\rme^{-V\circ\,\exp_{\gamma(t)}}\,\vol_{h_t}\mres B_\varepsilon^{h_t}(0),
\end{align}
For $\smash{\exp_{\gamma(t)}}$, a diffeomorphism on an open neighborhood of $\smash{B_\varepsilon^{h_t}(0)}$ in $T_{\gamma(t)}\mms$, set
\begin{align}\label{Eq:NuY}
    \nu_{\gamma(t)}^\varepsilon := (\exp_{\gamma(t)})_\push\xi_{\gamma(t)}^\varepsilon
\end{align}
the normalized $n$-dimensional weighted measure on the spacelike geodesic ball of radius $\varepsilon>0$ transverse to $\dot\gamma(t)$.

Supporting the reference measures on a \emph{codimension-one} slice transverse to the geodesic appears, to the best of our knowledge, not to have been considered before even in Riemannian signature. Conceptually, this is the more economical choice: since $\Ric(v,v)$ is the average of the sectional curvatures of the planes containing $v$, all of the relevant curvature is already carried by the directions orthogonal to $v$, and the radial direction contributes nothing at leading order. The codimension-one slice $\smash{v^\perp}$ is thus the minimal support on which the second-order asymptotics still detect $\Ric(v,v)$, and it reproduces \emph{verbatim} the expansion of the codimension-zero construction, the only trace of the lower-dimensional support being the value of the universal prefactor. However we introduced a new proof based on a quantitative Moser gradient method,  quantitative application of Sobolev estimates, and length expansion.

Our first theorem is referred as unweighted construction. Note that our measure is weighted,  we explain the quantifier ``unweighted'' a bit later. 

\begin{theorem}[Unweighted reconstruction, \cref{Th:MainInformal}] 
\label{thm:unweighted}
Under Conventions \ref{assumption1} and \ref{Re:SmallEnough}, the following holds for $\delta, \varepsilon >0$ sufficiently small with $\smash{\varepsilon=\rmo(\delta^{5/2})}$:
\begin{align}\label{expansion}
    \ell_1(\nu_x^\varepsilon,\nu_{\gamma(\delta)}^\varepsilon) = \delta\,\Big[1-\frac{\varepsilon^2}{2(n+2)}\,\Ric(v,v) + \rmO(\varepsilon^3)+\rmO(\delta\varepsilon^2)\Big],
\end{align}
where  $\smash{\ell_1}$ denotes the $1$-Lorentz--Wasserstein distance.
\end{theorem}

The Ricci curvature in the timelike direction $v$ is thus recovered, in a precise asymptotic regime, from the transport distance between two neighbouring slices. 
 The corresponding coarse timelike Ricci curvature $\smash{1-\ell_1/l}$ is asymptotically a positive multiple of $\Ric(v,v)$, exactly as in the Riemannian theorems of Ollivier \cite{ollivier2009} and Arnaudon--Li--Petko \cite{arnaudon-li-petko2025}. However, both work with non-degenerate measures on Riemannian balls, and in both cases $n$ is the dimension of the Riemannian manifold, while here 
 the measures are supported on the spacelike hypersurfaces of dimension $n$ in the $n+1$-dimensional manifold.
One innovation is that we used Nash gradient construction of optimal transport map to obtain key quantitative estimates, 
  while in Arnaudon--Li--Petko \cite{arnaudon-li-petko2025} uses an approximate explicit map.
 
 \begin{remark}
A formal comparison can be made with the coarse extrinsic
curvature of Arnaudon--Li--Petko
\cite{arnaudon-li-petko2025-ex}.
Their construction concerns an $m$-dimensional Riemannian
submanifold $N\subset\R^{m+k}$.
To distinguish their parameters from ours, denote their
tangential radius by $a$ and their normal radius by $\sigma$.
For $y=\exp^N_{x_0}(\delta e_1)$, their expansion reads
\begin{align*}
    W_1(\mu_{x_0}^{\sigma,a},\mu_y^{\sigma,a})
    &=\|y-x_0\|
    \left[
        1+\left(
            \frac{\sigma^2}{k+2}
            -\frac{a^2}{2(m+2)}
        \right)
        \left\langle
            \mathrm{II}_{x_0}(e_1,e_1),H(x_0)
        \right\rangle
    \right]
    +\rmO(\delta^4).
\end{align*}
Here $\mathrm{II}$ is the second fundamental form and
$H=\operatorname{tr}\mathrm{II}$ is the mean curvature vector,
without the normalizing factor $\f 1m$. Their test probability measures have $(m+k)$-dimensional
support in tubes of tangential radius $a$ and normal
radius $\sigma$. 
A more detailed discussion is presented at the end of the paper
\end{remark}

It is interesting to note that although the test measures involves $V$, the Taylor expansion does not see $V$, this is quite different from that in the Riemannian case \cite{arnaudon-li-petko2025}.  In our spacelike construction of the transported measures, the geodesic from $x$ to $\gamma(\delta)$ is orthogonal to both slices, so there is no first-order length contribution to carry the potential gradient; the weight reappears only through the temporal smearing of \cref{Th:WeightedReconstruction}, where the Hessian of $V$ enters via the second moment in the timelike direction.

Our second main result uses a \emph{weighted construction}, which is the natural home of the synthetic timelike curvature-dimension condition of Cavalletti--Mondino \cite{cavalletti-mondino2020} and the regime in which the reconstruction must ultimately be carried out. Smearing the spacelike slices into thin timelike tubes and calibrating the temporal against the spatial scale, see \cref{Re:WeightInvisible}.

To this end, we consider the the normalized weighted measure on denote the timelike tube in $\mms$ at $\gamma(t)$ of spatial thickness $\varepsilon$ and temporal height $\eta$, pushed forward by the exponential map $\Psi$ from the cylinder $ D_{\varepsilon,\eta}:= (-\eta,\eta)\times B_\varepsilon^{h_0}(0)\subset I\times v^\perp$. The exponential map $\Psi_\gamma\colon I\times D_{\varepsilon,\eta}\to \mms$  around the point $\gamma(t)$ is given by
$$   \Psi_{\gamma(t)}(s,w) := \exp_{\gamma(t)}(s\dot\gamma(t)+\PT{t}\,w).$$

More precisely define $$
    \mathscr{T}_{\gamma(t)}^{\varepsilon,\eta} := \Psi_{\gamma(t)}(D_{\varepsilon,\eta})
$$
and  the normalized measure
$$
    \mu_{\gamma(t)}^{\varepsilon,\eta} := \meas\big[\mathscr{T}_{\gamma(t)}^{\varepsilon,\eta}\big]^{-1}\,\meas\mres\mathscr{T}_{\gamma(t)}^{\varepsilon,\eta},
$$
where $\meas= \rme^{-V}\,\vol_g$. All discussions are local, and restricted to the tubular neighborhood $U$ of $\gamma$ in the Fermi-coordinate. The same construction recovers the damped Ricci curvature/Bakry--\'Emery tensor, $\Ric + \Hess V$, associated with a weighted measure $\smash{\rme^{-V}\vol_g}$.

\begin{theorem}[Weighted reconstruction, \cref{Th:WeightedReconstruction}]
Given $\varepsilon>0$, introduce
\begin{align*}
\eta(\varepsilon) :=\sqrt{ \frac{3}{2(n+2)}}\,\varepsilon    
\end{align*}
Assume Conventions \ref{assumption1} and \cref{Re:SmallEnough}. 
Then for every $\delta,\varepsilon>0$ sufficiently small with $\smash{\varepsilon=\rmo(\delta^{5/2})}$, we have
\begin{align*}
\ell_1(\mu_x^{\varepsilon,\eta(\varepsilon)},\mu_{\gamma(\delta)}^{\varepsilon,\eta(\varepsilon)}) &= \delta\,\Big[1 - \frac{\varepsilon^2}{2(n+2)}\big[\!\Ric(v,v) + \Hess V(v,v)\big]\\
    &\qquad\qquad + \rmO(\varepsilon^3) + \rmO(\delta\varepsilon^2)\Big].
\end{align*}
\end{theorem}
In particular denoting $\kappa_\varepsilon$  the the coarse timelike Ricci curvature induced by the family $\smash{\mu_\bullet^{\varepsilon,\eta(\varepsilon)}\colon \gamma(I)\to\Prob(\mms)}$, we obtain
\begin{align*}
\lim_{\substack{\delta,\varepsilon\to 0,\\\varepsilon=\rmo(\delta^{5/2})}} \frac{2(n+2)}{\varepsilon^2}\,\kappa_\varepsilon(x,\gamma(\delta)) = \Ric(v,v) + \Hess V(v,v).
\end{align*}

These two theorems are the main results of the paper.

After completing the proof of the main theorems in May 2026,  we have learned that a Lorentzian coarse Ricci  curvature recovering the timelike Ricci tensor from measures on nearby \emph{causal diamonds} has been obtained independently and contemporaneously by Barton--Borza--R\"ohrig \cite{barton-borza-rohrig2026+}, who also study the discrete causal set curvature and provide the numerical experiments mentioned above. Their result involves different reference measures and a different dimensional constant than ours; their proof follows the Riemannian argument of Arnaudon--Li--Petko \cite{arnaudon-li-petko2025} very closely. Our emphasis differs in two respects: we first work with codimension-one measures, an aspect to our best knowledge not explored before (with a proof that differs from \cite{arnaudon-li-petko2025,arnaudon-li-petko2025-ex} in that we construct an explicit transport map using Moser-type flows \cite{Moser1965,DacorognaMoser1990}), which yields the clean expansion above, and we establish the \emph{weighted} Ricci curvature (Bakry--\'Emery tensor) reconstruction. Moreover, as mentioned before, our codimension one construction does not recover the Hessian of the potential.

\subsection{Organization}\label{Sub:Organization} The paper is organized as follows. \cref{Ch:Prereq} fixes our standing assumptions, the conventions for the time separation function $l$ and the $1$-Lorentz--Wasserstein distance $\ell_1$, and recalls the construction of Fermi coordinates in a self-contained way. The technical heart of the asymptotic analysis is collected in \cref{Ch:Asymptotic}: Fermi metric expansion, the length expansion of a geodesic variation (\cref{Le:LengthVariation}), the averaging identity relating $\Ric$ to a Euclidean second moment, and the time-separation expansion between two transverse slice points. \cref{Ch:Det} carries the codimension-one transport-map construction, the explicit construction of a transport map, and the proof of the main \cref{Th:MainInformal}. The weighted Bakry--\'Emery extension via thin tubes is proved in \cref{Ch:Weighted}; the cylinder correction map is detailed in \cref{Le:CorrectionMapTube}.

\section{Prerequisites}\label{Ch:Prereq}

We employ Einstein's summation convention. Any pair of repeated Latin indices $i$, $j$, $k$, etc.~appearing in the same expression is summed over; these  sums range over $\{1,\dots,n\}$. Any pair of repeated Greek indices $\alpha$, $\beta$, $\gamma$, etc.~is summed over, but these sums range over $\smash{\{0,1,\dots,n\}}$. Here, $n\in\N$ will be such that the dimension of the manifold in question is $n+1$.

 If $(X,\tau)$ is a Polish topological space,  $\Prob(X)$ denotes the totality of all Borel probability measures on $X$. It is endowed with the narrow topology. The subscript $\smash{_\sharp}$ will denote the usual push-forward operation of measures. Recall that a \emph{coupling} of $\mu,\nu\in\Prob(X)$ is a measure $\smash{\pi\in\Prob(X^2)}$ which obeys $\smash{(\pr_1)_\push\pi = \mu}$ and $\smash{(\pr_2)_\push\pi = \nu}$; here, $\smash{\pr_i\colon\mms^2\to\mms}$ denotes the projection onto the $i$-th coordinate, where $i\in\{1,2\}$. We will write projections with more coordinates in the same way.
 
 In the next three subsections, we collect basic notions from spacetime geometry. The reader is referred to Hawking--Ellis \cite{hawking-ellis1973}, O'Neill \cite{oneill1983}, Wald \cite{wald1984}, or Beem--Ehrlich--Easley \cite{beem-ehrlich-easley1996} for details. Throughout, all geometric quantities are understood with respect to a fixed smooth Lorentzian structure.

\subsection{Spacetime}\label{Sub:Spacetime} 
Let $\mms$ be a smooth, connected, Hausdorff manifold, assumed to have dimension $n+1$, where $n\in\N$. The set of all smooth vector fields on $\mms$ is denoted $\mathfrak{X}(\mms)$. 
We fix a smooth \textit{Lorentzian metric} $g$ on $\mms$, that is, a smooth section of $\smash{(T^*\mms)^{\otimes 2}}$ which is symmetric and of constant signature $+,-,\dots,-$. A tangent vector $v\in T\mms$ is called \emph{timelike} if $g(v,v)>0$, \emph{lightlike} if $v\neq 0$ and $g(v,v)=0$, \emph{causal} if it is timelike or lightlike, and \emph{spacelike} if $v=0$ or $g(v,v)<0$. The causal character of a smooth vector field or a smooth curve is defined analogously. We assume a \textit{time orientation}, i.e.~a timelike vector field $S\in\mathfrak{X}(\mms)$, to be given; a causal vector $v\in T\mms$ is  \emph{future-directed} if $g(S,v)>0$ and \emph{past-directed} if $g(S,v)<0$. Since $S$ never used explicitly, we drop it from the notation and simply call the tuple  $(\mms,g)$, tacitly endowed with $S$, a \textit{space\-time}.

We assume throughout the spacetime $(\mms,g)$ is \emph{globally hyperbolic}. Equivalently, after the characterization of Bernal--Sánchez \cite{bernal-sanchez2007},  the causal relation $\smash{J_g}$ is anti\-symmetric and every \emph{causal diamond} $\smash{J_g^+(x)\cap J_g^-(y)}$ is  compact in $\mms$ for every $x,y\in\mms$. In particular, $(\mms,g)$ admits a smooth Cauchy temporal function $\phi$, whose gradient is timelike and the level sets of $\phi$ are smooth spacelike Cauchy hypersurfaces.
This hypothesis has some implications we tacitly use below. First, it ensures the existence of arbitrarily small convex neighborhoods.   Second, $l$ does not assume the value $\infty$ and is upper semicontinuous with continuous positive part. Third, given $x,y\in\mms$ with $\smash{y\in I_g^+(x)}$, the supremum in the definition of $l(x,y)$ is attained by a --- up to reparametrization --- (locally unique) future-directed timelike geo\-desic (maximizing by definition), and $l$ is smooth on a neighborhood of $(x,y)$.

Denote by $\smash{J_g^+(x)\subset\mms}$ the set of points reachable from $x\in\mms$ by future-directed causal curves, 
$\smash{J_g^-(x)\subset\mms}$ denotes points reachable from $x\in\mms$ by past-directed causal curves. The sets $\smash{I_g^\pm(x)\subset\mms}$ are defined analogously by replacing ``causal'' by ``timelike''. Finally, we set $J_g=\{(x,y)\in\mms : y \in J_g^+(x)\}$ and define $I_g$ analogously.

We write $\nabla$ for the \emph{Levi-Civita connection} of $g$, $\Rm$ for its \emph{Riemann curvature}, with the sign convention
\begin{align}
\Rm(X,Y)Z := \nabla_X\nabla_Y Z - \nabla_Y\nabla_X Z - \nabla_{[X,Y]}Z,
\end{align}
for every $X,Y,Z\in\mathfrak{X}(\mms)$, and $\Ric$ for the associated \emph{Ricci curvature}; here, $[\cdot,\cdot]$ means  the usual Lie bracket. If $\{e_0,e_1,\dots,e_n\}$ is a local orthonormal basis of $T_xM$, we define by $\Rm$ the Riemann tensor, whose components are
$$R_{\alpha\beta\gamma\delta}=g(\Rm(e_\alpha,e_\beta, e_\gamma,e_\delta),$$
and the symmetric matrix
\begin{equation}
    a_{ij} := -R_{i00j} \label{Eq:A(t)def}.
\end{equation}
   In particular, we have $\Ric(v,v)=a_{ij}\,\delta^{ij} $.

For a smooth curve $\smash{\gamma\colon I\to\mms}$, where $\smash{I\subset\R}$ is an interval,  and a time $t\in I$, \emph{parallel transport} along the curve $\gamma$ from $\gamma(0)$ to $\gamma(t)$ means the linear $g$-isometry $\smash{\PT{t}\colon T_{\gamma(0)}\mms\to T_{\gamma(t)}\mms}$. Given $x\in\mms$, let $\smash{\exp_x}$ be the \emph{exponential map} of $g$ at $x\in\mms$, defined on a neighbor\-hood of the origin in $\smash{T_x\mms}$.

\subsection{Causality and  time separation function}\label{Sub:Causality} Let $\smash{I_g,J_g\subset\mms^2}$ designate the usual relations of \emph{chronology} and \emph{causality}, respectively. Corresponding futures and pasts of points or subsets of $\mms$ are defined in the standard way.

Given any smooth future-directed causal curve $\gamma\colon I\to\mms$, where $I\subset\R$ is an inter\-val, we define its \emph{length} by
\begin{align}
    \Len(\gamma) := \int_I\sqrt{g(\dot\gamma(t),\dot\gamma(t))}\d t.
\end{align}
The \emph{time separation function}  $\smash{l\colon \mms^2\to\R\cup\{-\infty,\infty\}}$ is
\begin{align}
l(x,y) &:= \sup\lbrace\Len(\gamma) : \gamma\colon I\to\R \textnormal{ smooth future-directed }\\
&\qquad\qquad\textnormal{causal curve from }x\textnormal{ to }y\rbrace,
\end{align}
where we stipulate $\sup\emptyset:=-\infty$. It satisfies the reverse triangle inequality under the conventions $(-\infty)+a := a+(-\infty) := -\infty$ for every $a\in\R\cup\{-\infty,\infty\}$.

A \emph{geodesic} means a smooth solution $\smash{\gamma\colon I\to\mms}$ of the geodesic equation $\smash{\nabla_{\dot\gamma}\dot\gamma = 0}$ on $I$, where $I\subset\R$ is an interval. A future-directed time\-like geodesic $\smash{\gamma\colon[a,b]\to\mms}$, where $a,b\in\R$ with $a<b$, will be called \emph{maximizing} if  $\smash{l(\gamma(a),\gamma(b)) = \Len(\gamma)}$. An open subset $U\subset\mms$ is called \emph{normal neighborhood} of $x\in\mms$ if it is the diffeo\-morphic image of a star-shaped neighborhood of the origin in $T_x\mms$ under $\exp_x$. An open subset of $\mms$ is called \emph{convex neighborhood} if it is a normal neighborhood of all its points. Any pair of chronologically related points in a convex neighborhood is joined by a unique future-directed timelike geodesic.

\begin{convention}[Smallness of parameters]\label{Re:SmallEnough} All assertions in this paper of the form ``for $\delta>0$ and $\varepsilon>0$ sufficiently small'' are to be understood relative to a fixed normal neighborhood of a base point $x\in\mms$ with compact closure and a fixed future-directed unit timelike vector $v\in T_x\mms$. Concretely, we shrink $\varepsilon$ and $\delta$ so that, first, all exponential maps and inverse exponential maps appearing below are diffeomorphisms onto their images and all geodesics in question are maximizing and unique, and second, the Fermi chart of \cref{Sub:Fermi} covers all points under consideration. With the exception of \cref{Le:LengthVariation}, all implicit constants in the Landau symbols $\rmO$ depend only on this fixed neighborhood, the dimension of $\mms$, and a finite number of derivatives of $g$ (and of the weight $V$, where present) on it, but never on $\varepsilon$, $\delta$, the position $w$ ranging over a spacelike ball in $\smash{v^\perp}$, or the time-shift along the tube. (Relevant notation is introduced below.) We do not repeat these qualifications.
\end{convention}

\subsection{Geodesic congruences and embedded structure}

We briefly explain the geometric meaning of the terms appearing in the Raychaudhuri equation. Let $\gamma$ be a unit future-directed timelike geodesic. Extend $\dot \gamma$ to a a local unit timelike vector field $U$ on a neighborhood of $\gamma$, such that $g(U,U)=1$ and $\nabla_UU=0$. 
In the neighbourhood, we set
$\smash{U^\perp=\{X\in TM:g(X,U)=0\}}$. 
Let $\pi:TM\to U^\perp$
be the orthogonal projection, given by
$\pi Z=Z-g(Z,U)U$.
Next, we introduce the spatial covariant derivative $\nabla^\perp U:U^\perp\to U^\perp$ of $U$ given by
$$
\nabla_X^\perp U=\pi(\nabla_XU).
$$
Observe that $g(\nabla_XU,U)=Xg(U,U)/2=0$ and  for $X\in U^\perp$, $
\nabla_X^\perp U=\nabla_XU$. 
Define the following bilinear form $B$ for $\smash{X,Y\in U^\perp}$: 
$$
B(X,Y)=g(\nabla_X^\perp U,Y)=g(\nabla_XU,Y).
$$
The bilinear form $B$ is the analog of the second fundamental form associated with the normal direction $U$. Lastly, we set
\begin{align*}
    \theta &=\operatorname{tr}(\nabla^\perp U),\\ 
\sigma &=\operatorname{sym}B-\frac{\theta}{n}\,g\big\vert_{{U^\perp}^2},\\
\omega &=\operatorname{alt}B.
\end{align*}
Then $B$ has the decomposition
$\smash{B=\theta\, g\big\vert_{{U^\perp}^2}+\sigma+\omega}$; the Raychaudhuri equation for the timelike geodesic congruence generated by $U$ becoems
$$\rmd\theta(U)=-\frac1n\theta^2-\big\Vert\sigma\big\Vert_{h_0}^2+\big\Vert\omega\big\Vert_{h}^2-\operatorname{Ric}(U,U);$$
here $\|\cdot\|_{h}$ denotes the Hilbert--Schmidt norm of bilinear forms on $U^\perp$ with respect to the inner product $\smash{h_0 :=-g\big\vert_{{U^\perp}^2}}$.

\subsection{The \texorpdfstring{$1$}{1}-Lorentz--Wasserstein distance}\label{Sub:LW} We use the subsequent  Lorentzian analog of the classical $1$-Wasserstein distance, introduced by Eckstein--Miller \cite{eckstein-miller2017} and further studied in spacetime geometry by Suhr \cite{suhr2018-theory} and McCann \cite{mccann2020}.

We call a coupling $\pi$ of $\mu,\nu\in\Prob(\mms)$ \emph{causal} if $\pi[J]=1$ and \emph{chronological} if $\pi[I]=1$, respectively. 

\begin{definition}[$1$-Lorentz--Wasserstein distance]\label{Def:LW}
The \emph{$1$-Lorentz--Wasserstein distance} $\ell_1\colon\Prob(\mms)^2\to\R\cup\{-\infty,\infty\}$ is given by
$$
\ell_1(\mu,\nu)
:=
\sup\left\{
\int_{\mms^2} l(x,y)\,d\pi(x,y):
\pi\in\Pi_{\leq}(\mu,\nu)
\right\},
$$
where $\Pi_{\leq}(\mu,\nu)$ denotes the set of causal couplings of $\mu$ and $\nu$.
We use the convention $\sup\varnothing=-\infty$.
\end{definition}

\begin{definition}[Coarse timelike Ricci curvature]\label{Def:CTRC} Given a subset $X\subset\mms$, let $\nu_\bullet\colon X\to\Prob(\mms)$ be a family of probability measures indexed by $X$. The associated \emph{coarse time\-like Ricci curvature} $\smash{\kappa\colon X^2\cap I_g\to\R\cup\{-\infty,\infty\}}$ is defined by
\begin{align*}
    \kappa(x,y) := 1-\frac{\ell_1(\nu_x,\nu_y)}{l(x,y)}.
\end{align*}
\end{definition}

The Lorentzian transport distance $\ell_1$ inherits the reverse triangle inequality from $l$; see Eckstein--Miller \cite{eckstein-miller2017}*{Thm.~13}. With the above convention, $\ell_1(\mu,\nu)\geq 0$ if and only if $\mu$ and $\nu$ admit a causal coupling. A causal coupling attaining the supremum in the definition of $\ell_1(\mu,\nu)$ is called \emph{$\ell_1$-optimal}. If a chronological coupling of $\mu$ and $\nu$ exists, then $\ell_1(\mu,\nu)>0$, but the converse need not hold in general.
On the other hand, if $\mu,\nu\in\Prob(\mms)$ have compact support and satisfy
$\supp\mu\times\supp\nu\subset I$,
then every coupling of $\mu$ and $\nu$ is chronological. In this case, $\ell_1(\mu,\nu)>0$. Moreover, by global hyperbolicity and compactness of the supports, $\ell_1(\mu,\nu)<\infty$, and there exists an $\ell_1$-optimal coupling; see, for instance, McCann \cite{mccann2020}*{Prop.~2.9}.

\section{Fermi coordinates and asymptotic expansions}\label{Ch:Asymptotic}
 Let $\smash{v^\perp\subset T_x\mms}$ be the $g$-orthogonal complement of $v$, viz. 
\begin{align}\label{v-perp}
    v^\perp := \{w\in T_x\mms : g(v,w)=0\},
\end{align}
endowed with the inner product $h_0$ obtained by restricting $-g$, namely
$$h_0 := -g\big\vert_{{v^\perp}^2}.
$$
We denote by $\smash{\vert\cdot\vert_{h_0}}$ the corrsponding norm and by $\smash{B_\varepsilon^{h_0}(0)\subset v^\perp}$ the ball of radius $\varepsilon>0$ centered at the origin, and by $\smash{\vol_{h_0}}$ the volume measure.

\subsection{Recapitulation of Fermi coordinates}
\label{Sub:Fermi} 

We  now recall properties of Fermi coordinates in spacetime geometry, more detail can be found in  Manasse--Misner \cite{manasse-misner1963}. Throughout $I\subset \R$ is an interval with nonempty interior containing zero.

Fix $x\in\mms$ and a future-directed unit timelike vector $\smash{v\in T_x\mms}$. Let $\gamma$ be a smooth locally defined timelike geodesic, parametrised by proper time, with the initial conditions $\gamma(0)=x$ and $\dot \gamma(0)=v$:
$$
\gamma(t) := \exp_x(tv), \quad t\in I.$$
Choose $\smash{\{e_1,\dots,e_n\}}\subset \smash{v^\perp}$ such that $\delta(e_i, e_j)=-\delta_{i,j}$ and Parallel translate them along $\gamma$: $e_i(t) := \PT{t}\,e_i$. Together with $e_0(t)=\dot \gamma(t)= \PT{t}\,e_0$, they form the Fermi frame. The associated Fermi-coordinates are defined  by the Fermi coordinate map
$\Phi: I \times O\to M$, where $O$ is a geodesic ball in $\R^n$ and for $t\in I$ and $w\in O\subset\R^n$,
$$\Phi(t, w^1, \dots, w^n)=\exp_{\gamma(t)}(w^i e_i).$$

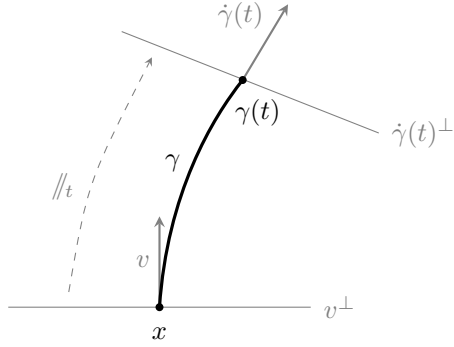
\begin{figure}[H]
\centering
\begin{tikzpicture}[>=stealth, scale=1]
  \draw[gray] (-2, 0) -- (2, 0);
  \node[gray, right] at (2.05, 0) {$v^\perp$};
  
  \draw[->, thick,gray] (0, 0) -- (0, 1.2);
  \node[left, gray] at (0, 0.6) {$v$};
  
  \draw[->, thick, gray] (1.1, 3) -- (1.7, 4);
  \node[left, gray] at (1.5, 3.85) {$\dot\gamma(t)$};
  
  \draw[gray] (-0.5, 3.64) -- (2.9, 2.30);
  \node[gray, right] at (2.95, 2.30) {$\dot\gamma(t)^\perp$};
  
  \draw[->, dashed, gray] (-1.2, 0.2) .. controls (-0.95, 1.65) .. (-0.1, 3.3);
  \node[gray, left] at (-1, 1.6) {$\PT{t}$};

  \draw[very thick] (0, 0) .. controls (0.1, 1.5) and (0.7, 2.5) .. (1.1, 3);
  \node at (0.175, 1.9) {$\gamma$};

  \fill (1.1, 3) circle (1.5pt);
  \node[above right=-1pt] at (0.9, 2.3) {$\gamma(t)$};

  \fill (0, 0) circle (1.5pt);
  \node[below] at (0, -0.15) {$x$};
\end{tikzpicture}
\caption{The Fermi setup at $x$ along the timelike geodesic $\gamma$ with initial velocity $v$. The dashed arrow indicates parallel transport from $\smash{v^\perp}$ to $\smash{\gamma(t)^\perp}$ along $\gamma$.}
\label{Fig:Fermi}
\end{figure}

Then the central geodesic $\gamma(t)$ corresponds to $\Phi(t,0)$ and the metric $\bar g=\Phi^*g$ pulled back by $\Phi$. Since $e_i(t)=d\Phi_{(t,0)}\partial_{w^i}$ for every $i\in\{1,\dots,n\}$, along the curve $\gamma$ the pull-back metric has components 
\begin{align*}
  \bar g_{00}(t,0)&=1,\\
  \bar g_{i0}(t,0) &=0,\\ 
  \bar g_{ij}(t,0)&=-\delta_{ij}  
\end{align*}
where $j=1, \dots, n,$. The first-order spatial derivatives of all components vanish; in particular, so do the Christoffel symbols (which constitute the first-order coefficients in the following Taylor expansions).

We then expand, in  \cref{Le:FermiMetric} below, the metric along $\bar g$ near the geodesic $\gamma$, corresponding expansion in $\smash{w= w^\alpha  e_\alpha\in v^\perp}$:
$$\bar g_{\alpha \beta}(t,w)=\bar g_{\alpha \beta}(t,0)+\f 12 \partial_k \partial_l \bar g_{\alpha \beta}(t,0)\,w^kw^l+O(\big|w\big|_{h_0}^3).$$

The next lemma expresses quadratic corrections in terms of the curvature, detailed computations  can be found in Manasse--Misner \cite{manasse-misner1963}*{p.~743} (with the opposite signature convention than ours). 
With the usual abuse of notation, from now on we let $\smash{g_{\alpha\beta}(t,w)}$ denote the components of $\bar g$ in the above Fermi chart. 

\begin{lemma}[Transversal Fermi metric expansion]\label{Le:FermiMetric}  
When $\smash{\vert w\vert_{h_0}}$ is sufficiently small,  the following expansions hold:
\begin{align*}
g_{00}(t,w) &= 1 - a_{ij}(t)\,w^iw^j + \rmO(\big\vert w\big\vert_{h_0}^3),\\
g_{0i}(t,w) &= \rmO(\big\vert w\big\vert^2_{h_0}),\\
g_{ij}(t,w) &= -\delta_{ij} + \rmO(\big\vert w\big\vert_{h_0}^2).
\end{align*}
\end{lemma}

The \emph{minus} sign in the expansion of $g_{00}$, is the Lorentzian focusing phenomenon: in directions of positive timelike sectional curvature, nearby timelike geodesics are refocused, and the proper time accumulated along the slanted comparison curve of \cref{Le:LengthExpansion} is \emph{shortened}. This is what eventually flips the sign of the second-order coefficient relative to the Riemannian expansions of Ollivier \cite{ollivier2009} and Arnaudon--Li--Petko \cite{arnaudon-li-petko2025}. 
Note also, When $\min (i, j)\not =0$,  the curvature term along a geodesic in direction of  $w$, appearing in the expansion of $g_{i,j}$,  is modulated by $|\dot w|$ and hence on a different scale from those in $g_{00}$.

The following is a straightforward consequence of smoothness of $\Rm$ and parallel transport along $\gamma$ over the interval $[0,t]$ or $[-t,0]$, on which $\gamma$ has $g$-length $t$.

\begin{lemma}[Longitudinal curvature expansion]\label{Le:Aexpansion} For $t\in I$ sufficiently close to zero, the time-dependent matrix $a$ satisfies
\begin{align*}
    a_{ij}(t) = a_{ij}+\rmO(t).
\end{align*}
\end{lemma}

\subsection{Averaged Ricci curvature}

Our aimed expansion of the $1$-Lorentz--Wasser\-stein distance will produce, at second order, a quadratic form in the spatial variable integrated with respect to the uniform measure on a spacelike ball. Observe that for  every $i\in\{1,\dots,n\}$, $a_{ii}(t)$ is the sectional curvature of the plane spanned by $e_i(t)$ and $\smash{\dot\gamma(t)}$, viz. $
    a_{ii}(t) = \varkappa(e_i(t),\dot\gamma(t))$.
 To recognize the resulting contraction as the Ricci curvature, we record the representation of the latter as an average of sectional curvatures over a Euclidean ball. The basic identity is the second-moment computation on an $n$-dimensional Euclidean ball; we omit the simple proof that uses polar coordinates and reflection-symmetry arguments.

\begin{lemma}[Moments of a Euclidean ball]\label{Le:SecondMoment} Let $\smash{B_\varepsilon\subset\R^n}$ designate the centered Euclidean ball of radius $\varepsilon>0$. Then for every $i,j,k\in\{1,\dots,n\}$,
\begin{align}
\Leb^n[B_\varepsilon]^{-1}\int_{B_\varepsilon} z^iz^j\d z &= \frac{\varepsilon^2}{n+2}\,\delta^{ij},\\
\int_{B_\varepsilon} z^iz^jz^k\d z &= 0.\label{Eq:ThirdMoment}
\end{align}
\end{lemma}

\begin{corollary}[Ricci curvature as average over a spacelike ball]\label{Co:RicciBall} For every $\varepsilon>0$,
\begin{align}
\vol_{h_0}\big[B_\varepsilon^{h_0}(0)\big]^{-1}\,a_{ij}\int_{B_\varepsilon^{h_0}(0)} w^iw^j\d\vol_{h_0}(w) = \frac{\varepsilon^2}{n+2}\,\Ric(v,v).
\end{align}
\end{corollary}

\begin{proof} Under the isometric identification of $\smash{(v^\perp,h_0)}$ with $(\R^{n},\hEucl)$ fixed in \cref{Sub:Fermi}, $\smash{B_\varepsilon^{h_0}(0)}$ becomes the centered Euclidean ball $B_\varepsilon$ and $\vol_{h_0}$ becomes the Lebesgue measure $\smash{\Leb^n}$. Applying \cref{Le:SecondMoment} yields
\begin{align}
\vol_{h_0}\big[B_\varepsilon^{h_0}(0)\big]^{-1}\,a_{ij}\int_{B_\varepsilon^{h_0}(0)} w^iw^j\d\vol_{h_0}(w) = \frac{\varepsilon^2}{n+2}\,a_{ij}\,\delta^{ij} = \frac{\varepsilon^2}{n+2}\,\Ric(v,v),
\end{align}
which is the desired identity.
\end{proof}

\begin{remark}[On the constant $n+2$]\label{Re:Constant} The constant $n+2$ in the above corollary has a different origin than the same dimensional constant appearing in the Riemannian expansions of Ollivier \cite{ollivier2009} and Arnaudon--Li--Petko \cite{arnaudon-li-petko2025}. In our case, it originates from the dimension of the spacelike slice on which $\smash{\vol_{h_0}}$ is supported. On the other hand, the reference measures in the mentioned  Riemannian expansions average over a full-dimensional ball, which would produce a constant $n+3$ if their dimension was set to be $n+1$, cf.~\cite{arnaudon-li-petko2025}*{Lem.~3.2}. An interesting link could be made to Arnaudon--Li--Petko \cite{arnaudon-li-petko2025-ex}, treating the geodesic as the backbone.
\end{remark}

\subsection{Length expansion in Fermi coordinates}\label{Sub:LengthVar}

\cref{Le:LengthVariation} gives a second-order expansion of the length of a smooth timelike curve as we vary its longitudinal and transversal parameters in Fermi coordinates along $\gamma$. This  Lorentzian analog of Arnaudon--Li--Petko \cite{arnaudon-li-petko2025}*{Prop.~2.7}  will be applied below in the proof below to several explicit comparison curves.

\begin{lemma}[Fermi length expansion]\label{Le:LengthVariation} Let $\gamma$ be a timelike curve and $U$ a fixed tubular precompact neighborhood of $\gamma$.
Let $\smash{c\colon[0,1]\to \mms}$ be a smooth curve contained in $U$. In Fermi co\-ordinates, we parametrize its time by a constant speed $\zeta>0$ and write $c(\tau) = (\zeta\tau + a,w(\tau))$, where $a\in\R$. Set
    \begin{align}\label{Eq:hdotwinfty}
    r &:= \sup\{\big\vert w(\tau)\big\vert_{h_0}:\tau\in [0,1]\},\\
    \sigma &:= \sup\{\big\vert\dot w(\tau)\big\vert_{h_0} : \tau\in[0,1]\}
\end{align}
 Assume that $\zeta >0$ and $\varepsilon >0$ are sufficiently small so that $c$ is timelike. Suppose that $r=\rmO(\varepsilon)$ and $\smash{\sigma= \rmo(\zeta)}$. Then,
\begin{align*}
\Len(c) &= \zeta - \frac{1}{2\zeta}\int_{[0,1]}\big\vert\dot{w}(\tau)\big\vert_{h_0}^2\d\tau  - \frac{\zeta}{2}\int_{[0,1]}  a_{ij}(t(\tau))\,w^i(\tau)w^j(\tau)\d\tau\\
&\qquad\qquad + \rmO(\zeta\varepsilon^3) + \rmO(\varepsilon^2\sigma) + \rmO(\zeta^{-3}\sigma^4).
\end{align*}
\end{lemma}
\begin{proof}[Proof of \cref{Le:LengthVariation}] By \cref{Le:FermiMetric}, uniformly on $[0,1]$ we have
\begin{align*}
g(\dot c,\dot c) &= g_{00}(t,w)\,\dot t^2 + 2g_{0i}(t,w)\,\dot t\,\dot{w}^i + g_{ij}(t,w)\,\dot w^i\dot w^j\\
&= \big[1 - (a_{ij}\circ t)\,w^iw^j\big]\,\dot t^2 - \big\vert \dot w\big\vert_{h_0}^2 \\
&\qquad\qquad+ \rmO(\big\vert w\big\vert_{h_0}^3\,\dot t^2) + \rmO(\big\vert w\big\vert_{h_0}^2\,\dot t\,\big\vert \dot w\big\vert_{h_0}) + \rmO(\big\vert w\big\vert_{h_0}^2\,\big\vert \dot w\big\vert_{h_0}^2)\\
&= \zeta^2\,\big[\big[1 - (a_{ij}\circ t)\,w^iw^j\big] - \zeta^{-2}\,\big\vert \dot w\big\vert_{h_0}^2 + \rmO(\varepsilon^3)\\
&\qquad\qquad+ \rmO(\varepsilon^2\zeta^{-1}\sigma) + \rmO(\varepsilon^2\zeta^{-2}\sigma^2)\big];
\end{align*}
here, we have used $\smash{\vert w\vert_{h_0}=\rmO(\varepsilon)}$ uniformly on $[0,1]$, the identity $\smash{\dot t = \zeta}$ on $[0,1]$, and the uniform bound $\smash{\vert\dot w\vert_{h_0}\leq\sigma}$ on $[0,1]$. In short,
\begin{align*}
    g(\dot c,\dot c) = \zeta^2\,(1+\beta)
\end{align*}
where
\begin{align*}
    \beta := -(a_{ij}\circ t)\,w^iw^j-\zeta^{-2}\,\big\vert\dot w\big\vert_{h_0}^2 + \rmO(\varepsilon^3) + \rmO(\varepsilon^2\zeta^{-1}\sigma) + \rmO(\varepsilon^2\zeta^{-2}\sigma^2);
\end{align*}
in particular, $c$ is timelike if $\vert \beta\vert <1$ uniformly on $[0,1]$, which we will now verify. It is clear that the remainder terms lie in $\rmo(1)$ uniformly on $[0,1]$ by our hypotheses $\zeta^{-1}\sigma=\rmo(1)$ and $\varepsilon =\rmo(1)$. By \cref{Le:Aexpansion}, we have $\smash{a_{ij}\circ t =\rmO(1)+\rmO(t)}$. By the local uniform boundedness of $|w|_{h_0}$ and $\smash{\vert\dot w\vert_{h_0}=\rmO(\sigma)}$ on $[0,1]$ and the assumption $\smash{\zeta^{-2}\sigma^2 = \rmo(1)}$, 
\begin{align*}
    -(a_{ij}\circ t)\,w^iw^j -\zeta^{-2}\,\big\vert\dot w\big\vert_{h_0}^2 \geq -\frac{1}{2}.
\end{align*}
In fact, this argument shows the stronger asymptotic $\beta = \rmo(1)$ uniformly on $[0,1]$.

We now address the claimed expansion of $\Len(c)$. Taylor expanding the term $\smash{\sqrt{1+\beta}}$ for $\beta=\rmo(1)$ we infer, uniformly on $[0,1]$,
\begin{align}\label{Eq:GExp}
\begin{split}
    \sqrt{g(\dot c,\dot c)} &= \zeta\sqrt{1+\beta}\\
    &= \zeta + \frac{\zeta}{2}\,\beta + \rmO(\zeta\,\beta^2)\\
    &= \zeta - \frac{\zeta}{2}\,(a_{ij}\circ t)\,w^iw^j - \frac{1}{2\zeta}\,\big\vert \dot w\big\vert_{h_0}^2 \\
    &\qquad\qquad + \rmO(\zeta\varepsilon^3) + \rmO(\varepsilon^2\sigma) + \rmO(\zeta^{-1}\varepsilon^2\sigma^2) + \rmO(\zeta^{-3}\sigma^4).
    \end{split}
\end{align}
 Observing  $\smash{\rmO(\zeta^{-1}\varepsilon^2\sigma^2)}$ can be absorbed into $\smash{\rmO(\varepsilon^2\sigma)}$ since $\sigma=\rmo(\zeta)$, integrating the resulting expansion on $[0,1]$ yields the claim.
\end{proof}

\subsection{Time separation along a slanted comparison curve}\label{Sub:TimeSep}

 Let $\delta>0$ be sufficiently small, and $w,w'\in v^\perp$ transverse to a geodesic $\gamma$. 
The only analytic estimate input feeding both the upper and the lower bound of our main theorem  is
a second-order expansion of $\smash{l(\exp_x(w),\exp_{\gamma(\delta)}(\PT{\delta}\,w'))}$.  We prove it by applying \cref{Le:LengthVariation} to an affine comparison curve in Fermi coordinates and matching it from above by the second variation of arclength around the maximizing geodesic.

\begin{lemma}[Asymptotic lower bound on time separation]\label{Le:LengthExpansion} 
Let $\delta>0$ and $\varepsilon>0$ be sufficiently small. Assume the displacement bound
\begin{align}\label{Eq:DisplBound}
    \big\vert w-w'\big\vert_{h_0} = \rmO(\delta\varepsilon^2)
\end{align}
holds uniformly in $w,w'\in B_\varepsilon^{h_0}(0)$. Then
\begin{align}
l(\exp_x(w),\exp_{\gamma(\delta)}(\PT{\delta}\,w')) \geq \delta - \frac\delta2\,a_{ij}\,w^iw^j + \rmO(\delta\varepsilon^3) + \rmO(\delta^2\varepsilon^2).
\end{align}
\end{lemma}

\begin{proof} We connect $\smash{\exp_x(w)}$ to $\smash{\exp_{\gamma(\delta)}(\PT{\delta}\,w')}$ by an appropriate timelike curve and estimate its length from below; by definition of $l$, any such curve provides a lower bound for their time separation. 

Define the comparison curve $c\colon[0,1]\to\mms$ by its Fermi coordinates along $\gamma$ as $c(\tau) := (t(\tau),w(\tau))$, where
\begin{align}\label{Eq:AffRenPath}
\begin{split}
t(\tau) &:= \delta\tau,\\
w(\tau) &:= w + \tau\,(w'-w).
\end{split}
\end{align}
Since $w(0)=w$ and $\smash{w(1)=w'}$ while $\gamma$ has unit speed, we have $\smash{c(0)=\exp_x(w)}$ and $\smash{c(1) = \exp_{\gamma(\delta)}(\PT{\delta}\,w')}$. The displacement bound assumption implies
\begin{align}\label{Eq:SigmaBound}
    \sigma = \sup\{\big\vert \dot w(\tau)\big\vert_{h_0}: \tau \in[0,1]\} = \big\vert w'-w\big\vert_{h_0} = \rmO(\delta\varepsilon^2).
\end{align}
Apply \cref{Le:LengthVariation} with $\zeta = \delta$,  one has \begin{align}\label{Eq:LenC}
\begin{split}
l(\exp_x(w),\exp_{\gamma(\delta)}(\PT{\delta}\,w')) &\geq \Len(c)\\
&= \delta - \frac{1}{2\delta}\int_{[0,1]} \big\vert \dot w(\tau)\big\vert_{h_0}^2\d\tau\\
&- \frac{\delta}{2}\int_{[0,1]}  a_{ij}(t(\tau))\,w^i(\tau)w^j(\tau)\d\tau\\
&\qquad
 + \rmO(\delta\varepsilon^3) + \rmO(\varepsilon^2\sigma) + \rmO(\delta^{-3}\sigma^4).
 \end{split}
\end{align}
By \eqref{Eq:SigmaBound},   $\sigma = \rmO(\delta\varepsilon^2)$,
the last two remainders $\rmO(\varepsilon^2\sigma)$ and $\rmO(\delta^{-3}\sigma^4)$ are controlled by $\rmO(\delta\varepsilon^3)$. 
The kinetic integral cam be estimated directly:
\begin{align*}
\frac{1}{2\delta}\int_{[0,1]} \big\vert \dot w(\tau)\big\vert_{h_0}^2\d\tau = \frac{1}{2\delta}\,\big\vert w'-w\big\vert_{h_0}^2 = \rmO(\delta\varepsilon^4).
\end{align*}
It remains to address the curvature integral. Using $ a_{ij}(t(\tau)) = a_{ij} + \rmO(\delta)$ from\cref{Le:Aexpansion},
\begin{align*}
\frac{\delta}{2}\int_{[0,1]} a_{ij}(t(\tau))\,w^i(\tau)w^j(\tau)\d\tau = \frac{\delta}{2}\,a_{ij}\,w^iw^j + \rmO(\delta^2\varepsilon^2) + \rmO(\delta^2\varepsilon^3).
\end{align*}
Plugging in the uniform estimates $w^iw^j = \rmO(\varepsilon^2)$ and
$$    w^i(\tau)w^j(\tau) = w^iw^j + \rmO(\big\vert w\big\vert_{h_0}\,\big\vert w'-w\big\vert_{h_0})
= w^iw^j + \rmO(\delta\varepsilon^3),$$
we obtain
\begin{align*}
\frac{\delta}{2}\int_{[0,1]} a_{ij}(t(\tau))\,w^i(\tau)w^j(\tau)\d\tau = \frac{\delta}{2}\,a_{ij}\,w^iw^j + \rmO(\delta^2\varepsilon^2) + \rmO(\delta^2\varepsilon^3).
\end{align*}
The last remainder is absorbed by $\rmO(\delta\varepsilon^3)$ concluding the proof.
\end{proof}

The corresponding upper bound from \cref{Le:UpperSep} will require us to work directly with a maximizer of the time separation in question. To apply \cref{Le:LengthVariation}, we will affinely parametrize its longitudinal component in Fermi coordinates along $\gamma$. On the other hand, unlike the previous proof we have no freedom of choosing its transversal component; nevertheless, the next lemma establishes a qualitative control compared to an affine reference path as in \eqref{Eq:AffRenPath}.

\begin{lemma}[A priori estimates for geodesics in Fermi coordinates]\label{Le:GeodAPrioriEstimates}
Let $\delta>0$ and $\varepsilon>0$ be sufficiently small with $\varepsilon=\rmO(\delta)$. Let $\gamma$ denote a geodesic and $U$ a Fermi-tubular neighbourhood.
Let $\alpha\colon[0,1]\to\mms$ be a smooth future-directed timelike geodesic in $U$. Write, in the Fermi coordinates, $\alpha(\tau) = (t(\tau),w(\tau))$, 
 where  $t(\tau)=\delta\tau$ and $\smash{w\colon[0,1]\to v^\perp}$ is smooth with endpoints in $\smash{B_\varepsilon^{h_0}(0)}$. Define $\smash{z\colon[0,1]\to v^\perp}$ by
\begin{align*}
    z(\tau) := w(0)+\tau\,(w(1)-w(0)),
\end{align*}
where $v^\perp$ is identified with $\R^n$. Then we have
\begin{align}\label{Eq:AprioriBounds}
\begin{split}
\sup\{\big\vert\dot w(\tau)\big\vert_{h_0}:\tau\in[0,1]\} &= \rmO(\varepsilon),\\
\sup\{\big\vert\ddot w(\tau)\big\vert_{h_0}:\tau\in[0,1]\} &= \rmO(\delta^2\varepsilon),\\
\sup\{\big\vert w(\tau) - z(\tau)\big\vert_{h_0}:\tau\in[0,1]\} &= \rmO(\delta^2\varepsilon),
\end{split}
\end{align}
where the implicit constants are uniform in the endpoints of $w$.
\end{lemma}

\begin{proof} 
Let $I\times W$, where $W\subset v^\perp$,  be in the domain of the Fermi-coordinate and ${W'}=\cup_{w\in W} B_1(w)$ which is contained entirely on $v^\perp$ by the smallness assumption on the parameters.
In the sequel, we tacitly identify each fiber of $\smash{Tv^\perp}$ with $v^\perp$.  By the vanishing of $\ddot z$, and the equality of endpoints of $w$ and $z$, the map $\smash{u\colon [0,1]\to v^\perp}$ with $    u:= w-z$
solves the geodesic equation which in Fermi coordinates
\begin{align}\label{Eq:ODE}
    \ddot u = F_\delta(\bullet,u,\dot u)\quad\textnormal{on }[0,1]
\end{align}
with Dirichlet boundary conditions and where $F_\delta$ are given below.
An equivalent fixed point characterization is the following. Consider the Banach space $(X,\Vert\cdot\Vert)$ of functions in $\smash{C^1([0,1];v^\perp)}$ with Dirichlet boundary conditions and norm
\begin{align*}
    \Vert y\Vert := \big\Vert y\big\Vert_{C^0([0,1];\R^n)} + \big\Vert\dot y\big\Vert_{C^0([0,1];\R^n)}.
\end{align*}
Then $u$ is a fixed point of the map $\Phi\colon X\to X$ given by
\begin{align*}
    \Phi[y](\tau) := -\int_{[0,1]}G(\tau,s)\,F_\delta(s,y(s),\dot y(s))\d s,
\end{align*}
where $G\colon[0,1]^2\to\R$ denote the Green's function for the Laplacian on $[0,1]$ with Dirichlet boundary conditions given by $\smash{G(\tau,s) := \min\{s(1-\tau),\tau(1-s)\}}$. This will be our starting point to derive the claimed estimates for $u$.

The smooth functions $F_\delta\colon [0,1]\times W'\times v^\perp\to v^\perp$ are given by
\begin{align*}
    F_\delta^k(\tau,a,b) &:= -\Gamma_{00}^k(t(\tau),z(\tau)+a)\,\delta^2\\
    &\qquad\qquad - 2\delta\,\Gamma_{0i}^k(t(\tau),z(\tau)+a)\,\big[\dot z^i(\tau)+b^i\big]\\
    &\qquad\qquad - \Gamma_{ij}^k(t(\tau),z(\tau)+a)\,\big[\dot z^i(\tau)+b^i\big]\big[\dot z^j(\tau)+b^j\big];
\end{align*}
Since  $w(0),w(1) \in \smash{B_\varepsilon^{h_0}(0)}$, $\smash{\vert z\vert_{h_0}\leq\varepsilon}$ and $\smash{\vert\dot z\vert_{h_0}\leq 2\varepsilon}$ on $[0,1]$., together with the Lipschitz continuity of the Christoffel symbols on $I\times W'$, \begin{align}\label{Eq:Scales}
 \sup_{a\in W', b\in v^\perp} \sup_{\tau\in[0,1]}  \big\vert F_\delta(\tau,a,b)\big\vert_{h_0} \leq C\,\big[\varepsilon + \big\vert a\big\vert_{h_0}\big]\,\big[\delta+\varepsilon+\big\vert b\big\vert_{h_0}\big]^2;
\end{align}
the constant $C$ may and will be chosen uniformly in $\delta\in(0,\delta_0)$ and $\varepsilon\in(0,\varepsilon_0)$, where $\delta_0\in(0,\rho_0)$ and $\varepsilon_0\in (0,\rho_0)$ are sufficiently small and fixed. The same Lipschitz bounds on the Christoffel symbols also give a constant $C'>0$, uniform in the same sense, such that for every $\tau\in[0,1]$, every $a,a'\in W'$ with $\vert a\vert_{h_0},\vert a'\vert_{h_0}\leq\rho_0$, and every $b,b'\in v^\perp$ with $\vert b\vert_{h_0},\vert b'\vert_{h_0}\leq\rho_0$, 
\begin{align}\label{Eq:Lipschitzzzz}
    \big\vert F_\delta(\tau,a,b)-F_\delta(\tau,a',b')\big\vert_{h_0} \leq C'\rho_0\,\big[\big\vert a-a'\big\vert_{h_0}+\big\vert b-b'\big\vert_{h_0}\big].
\end{align}

By smooth dependence of the geodesic equation on their end points,  $$\Vert u\Vert+\| \dot  u\| \leq\rho_0$$ provided that  $\delta_0$ and $\varepsilon_0$ are sufficiently small.  Without loss of generality, we may and will choose $\rho_0$ such that $C'\rho_0/2 \leq 1/4$. It is straightforward to verify
\begin{align}\label{Eq:GreenInt}
\begin{split}
    \sup\!\Big\lbrace\!\int_{[0,1]} G(\tau,s)\d s : \tau\in[0,1]\Big\rbrace &= \frac{1}{8},\\
    \sup\!\Big\lbrace\!\int_{[0,1]} \Big\vert \frac{\rmd}{\rmd\tau}G(\tau,s)\Big\vert \d s : \tau\in[0,1]\Big\rbrace &= \frac{1}{2}.
    \end{split}
\end{align}
Thus, for every $\tau\in[0,1]$,
\begin{align}\label{Eq.Takesup}
\begin{split}
    \big\vert u(\tau)\big\vert_{h_0} + \big\vert\dot u(\tau)\big\vert_{h_0} &= \big\vert \Phi[u](\tau)\big\vert_{h_0} + \big\vert\dot\Phi[u](\tau)\big\vert_{h_0}\\
    &\leq \big\vert \Phi[u](\tau) -\Phi[0](\tau)\big\vert_{h_0} + \big\vert\Phi[0](\tau)\big\vert_{h_0}\\
    &\qquad\qquad+ \big\vert \dot\Phi[u](\tau) -\dot\Phi[0](\tau)\big\vert_{h_0} + \big\vert\dot\Phi[0](\tau)\big\vert_{h_0}.
    \end{split}
\end{align}
We now estimate these four summands separately.
For the first summand, we use \eqref{Eq:Lipschitzzzz} and \eqref{Eq:GreenInt} to get
\begin{align*}
    \big\vert \Phi[u](\tau) -\Phi[0](\tau)\big\vert_{h_0} &\leq \int_{[0,1]}G(\tau,s)\,\big\vert F_\delta(s,u(s),\dot u(s)) - F_\delta(s,0,0)\big\vert_{h_0}\d s\\
    &\leq C'\rho_0\int_{[0,1]}G(\tau,s)\,\big[\big\vert u(s)\big\vert_{h_0} +\big\vert \dot u(s)\big\vert_{h_0}\big]\d s\\
    &\leq C'\rho_0\,\Vert u\Vert\int_{[0,1]}G(\tau,s)\d s\\
    &\leq \frac{1}{4}\,\Vert u\Vert.
\end{align*}

Analogously, the third summand is estimated by
\begin{align*}
    \big\vert \dot \Phi[u](\tau) -\dot \Phi[0](\tau)\big\vert_{h_0} &\leq \int_{[0,1]}\Big\vert\frac{\rmd}{\rmd\tau}G(\tau,s)\Big\vert\,\big\vert F_\delta(s,u(s),\dot u(s)) - F_\delta(s,0,0)\big\vert_{h_0}\d s\\
    &\leq C'\rho_0\,\Vert u\Vert\int_{[0,1]}\Big\vert\frac{\rmd}{\rmd\tau}G(\tau,s)\Big\vert\d s\\
    &\leq \frac{1}{4}\,\Vert u\Vert.
\end{align*}

By \eqref{Eq:Scales} and again \eqref{Eq:GreenInt}, the second and the fourth summand are easily seen to be of order $\rmO(\delta^2\varepsilon)$, respectively.

Taking the supremum over $\tau\in[0,1]$ in \eqref{Eq.Takesup} thus yields
\begin{align*}
    \Vert u\Vert \leq \frac{1}{2}\,\Vert u\Vert + \rmO(\delta^2\varepsilon)
\end{align*}
and consequently
\begin{align*}
    \Vert u\Vert =\rmO(\delta^2\varepsilon).
\end{align*}
Since $\ddot z$ vanishes identically, this accomplishes the last two claims from \eqref{Eq:AprioriBounds}. Lastly, this also yields, uniformly in $\tau\in[0,1]$ and in the endpoints $w(0),w(1)\in B_\varepsilon^{h_0}(0)$,
\begin{align*}
    \big\vert\dot w(\tau)\big\vert_{h_0} \leq \big\vert\dot z(\tau)\big\vert_{h_0} + \big\vert\dot u(\tau)\big\vert_{h_0} = \big\vert w(1)-w(0)\big\vert_{h_0} + \rmO(\delta^2\varepsilon) = \rmO(\varepsilon).
\end{align*}
This concludes the proof.
\end{proof}

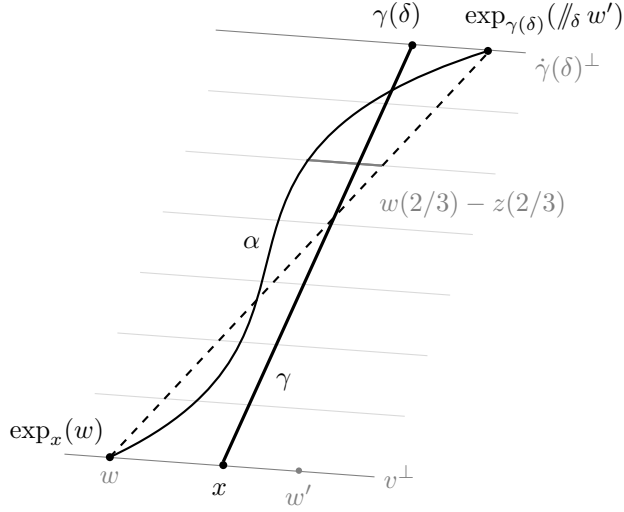
\begin{figure}[H]
\centering
\begin{tikzpicture}[>=stealth, scale=1]
\draw[thick, gray] (2.1,3.785) -- (3.1,3.71);
  \draw[gray!30] (-0.7, 0.7) -- (3.4, 0.4);
  \draw[gray!30] (-0.4, 1.5) -- (3.7, 1.2);
  \draw[gray!30] (-0.1, 2.3) -- (4.0, 2.0);
  \draw[gray!30] (0.2, 3.1) -- (4.3, 2.8);
  \draw[gray!30] (0.5, 3.9) -- (4.6, 3.6);
  \draw[gray!30] (0.8, 4.7) -- (4.9, 4.4);
  \draw[thick, gray] (2.1,3.785) -- (3.1,3.71);
  \draw[gray] (-1.1, -0.1) -- (3, -0.4);
  \node[gray, right] at (3, -0.4) {$v^\perp$};
  
  \fill (-0.5, -0.14) circle (1.5pt);
  \node[gray,below=2pt] at (-0.5, -0.14) {$w$};
  \fill (1.0, -0.245) circle (1.5pt);
  \node[below=2pt] at (0.95, -0.3) {$x$};
  \fill[gray] (2.0, -0.318) circle (1.2pt);
  \node[gray,below=2pt] at (2.0, -0.31) {$w'$};
  \node at (-1.2,0.2) {$\exp_x(w)$};
  \node at (5.25,5.65) {$\exp_{\gamma(\delta)}(\PT{\delta}\,w')$};
  
  \draw[gray] (0.9, 5.5) -- (5, 5.2);
  \node[gray, right] at (5, 5) {$\dot\gamma(\delta)^\perp$};
  
  \fill (4.5, 5.225) circle (1.5pt);
  
  \draw[thick] (-0.5, -0.14) .. controls (3, 1.5) and (0, 3.8) .. (4.5, 5.225);
  \node at (1.375, 2.7) {$\alpha$};
  
  \draw[thick, dashed] (-0.5, -0.14) -- (4.5, 5.2);
  \draw[very thick] (1, -0.2) -- (3.5, 5.3);
  \fill (3.5, 5.3) circle (1.5pt);
\node[above] at (3.3, 5.4) {$\gamma(\delta)$};

  \node at (1.8,0.85) {$\gamma$};
  \node[gray] at (4.3,3.2) {$w(2/3)-z(2/3)$};
\end{tikzpicture}

\caption{The setup of \cref{Le:UpperSep}. The timelike geodesic $\alpha$ is written in Fermi coordinates along $\gamma$ with transversal component $w$, with prescribed endpoints in the spacelike $\varepsilon$-balls on $\smash{v^\perp}$ and $\smash{\dot\gamma(\delta)^\perp}$. The affine reference path whose transversal component is $z$ is shown dashed. The faint gray lines indicate the foliation by the slices $\smash{\dot\gamma(t)^\perp}$ for intermediate $t\in[0,\delta]$.}
\label{Fig:Geodesic}
\end{figure}

We stress the nonpositive summand $\smash{-\vert w-w'\vert_{h_0}^2/2\delta}$ in the formula below,  which will be ignored in \cref{Co:TimeSepExpansion}.

\begin{lemma}[Asymptotic upper bound on time separation]\label{Le:UpperSep} Assume $\delta>0$ and $\varepsilon>0$ are sufficiently small with $\smash{\varepsilon = \rmo(\delta^{5/2})}$. Then, uniformly in $\smash{w,w'\in B_\varepsilon^{h_0}(0)}$,
\begin{align*}
l(\exp_x(w),\,\exp_{\gamma(\delta)}(\PT{\delta}\,w')) &\leq \delta - \frac{1}{2\delta}\,\big\vert w-w'\big\vert_{h_0}^2 - \frac\delta2\,a_{ij}\,w^iw^j\\
&\qquad\qquad + \rmO(\delta\varepsilon^3) + \rmO(\delta^2\varepsilon^2).
\end{align*}
\end{lemma}

\begin{proof} 

Let $\alpha$  be as in \cref{Le:GeodAPrioriEstimates}.
 \cref{Le:LengthVariation} entails
\begin{align*}
l(\exp_x(w),\exp_{\gamma(\delta)}(\PT{\delta}\,w')) &= \Len(\alpha)\\
&= \delta - \frac{1}{2\delta}\int_{[0,1]} \big\vert\dot w(\tau)\big\vert_{h_0}^2\d\tau\\
&\qquad\qquad - \frac{\delta}{2}\int_{[0,1]} a_{ij}(t(\tau))\,w^i(\tau)w^j(\tau)\d\tau\\
&\qquad\qquad + \rmO(\delta\varepsilon^3) + \rmO(\varepsilon^2\sigma) + \rmO(\delta^{-3}\sigma^4).
\end{align*}
The Cauchy--Schwarz inequality implies
\begin{align*}
-\frac{1}{2\delta}\int_{[0,1]} \big\vert \dot w(\tau)\big\vert_{h_0}^2\d\tau \leq -\frac{1}{2\delta}\Big\vert\!\int_{[0,1]} \dot w(\tau) \d\tau\Big\vert_{h_0}^2 = -\frac{1}{2\delta}\big\vert w-w'\big\vert_{h_0}^2.
\end{align*}

We treat each contribution separately. Firstly, the displayed remainders are swallowed by $\smash{\rmO(\delta\varepsilon^3)+\rmO(\delta^2\varepsilon^2)}$. To see this, first, \cref{Le:GeodAPrioriEstimates} implies $\sigma = \rmO(\varepsilon)$, whence $\varepsilon^2\sigma = \rmO(\varepsilon^3)$; since $\varepsilon =\rmo(\delta^{5/2})$ by assumption, we have $\varepsilon^3/(\delta^2\varepsilon^2) = \varepsilon/\delta^2 =\rmo(1)$ and thus $\varepsilon^2\sigma = \rmO(\delta^2\varepsilon^2)$. Second, since $\delta^{-3}\varepsilon^4/(\delta^2\varepsilon^2) = \varepsilon^2/\delta^5=\rmo(1)$ again as hypothesized, we have $\delta^{-3}\sigma^4=\rmO(\delta^2\varepsilon^2)$.

For the curvature integral, let $z$ be as in \cref{Le:GeodAPrioriEstimates}. The last identity therein in combination with  $\vert w\vert_{h_0} = \rmO(\varepsilon)$ as observed above in verifying the hypotheses of \cref{Le:LengthVariation} entails for every $i,j\in\{1,\dots,n\}$, uniformly in $\tau\in[0,1]$,
\begin{align}\label{Eq:wiwjzz}
w^i(\tau)w^j(\tau) - z^i(\tau)z^j(\tau) = \rmO(\delta^2\varepsilon^2).
\end{align}
The same bookkeeping as in the proof of \cref{Le:LengthExpansion}, applied to $z$, gives 
\begin{align*}
\frac{\delta}{2}\int_0^1 a_{ij}(t(\tau))\,z^i(\tau)z^j(\tau)\d\tau = \frac\delta 2\, a_{ij}\,w^iw^j + \rmO(\delta^2\varepsilon^2) + \rmO(\delta\varepsilon^3);    
\end{align*}
replacing $z$ by $w$ contributes only an error in $\rmO(\delta^3\varepsilon^2)$ by (\ref{Eq:wiwjzz}), which is swallowed by $\smash{\rmO(\delta^2\varepsilon^2)}$. Collecting these observations concludes the proof.
\end{proof}

\begin{remark}[Joint scaling regime]\label{Re:Regime} The main reconstruction \cref{Th:MainInformal,Th:WeightedReconstruction} will be stated in the joint regime $\varepsilon = \rmo(\delta^{5/2})$ also hypothesized by \cref{Le:UpperSep}. Here, the exponent $5/2$ is dictated by the requirement that the Lagrange remainder $\smash{\rmO(\delta^{-3}\varepsilon^4)}$ entering the asymptotic upper bound \cref{Le:UpperSep} be absorbed into the displayed $\smash{\rmO(\delta^2\varepsilon^2)}$. The lower bound \cref{Le:LengthExpansion}, by contrast, only requires the weaker regime $\varepsilon = \rmo(\delta)$ thanks to the displacement bound (\ref{Eq:DisplBound}), which will be realized by the correction map we construct in \cref{Le:CorrectionMap} below.
\end{remark}

Combining \cref{Le:LengthExpansion,Le:UpperSep} readily yields the following.

\begin{corollary}[Taylor expansion of time separation between transverse points]\label{Co:TimeSepExpansion}
Assume $\delta>0$ and $\varepsilon>0$ are sufficiently small with $\smash{\varepsilon = \rmo(\delta^{5/2})}$. Then, uniformly in $\smash{w,w'\in B_\varepsilon^{h_0}(0)}$ subject to the displacement bound
\begin{align}
    \big\vert w-w'\big\vert_{h_0} = \rmO(\delta\varepsilon^2),
\end{align}
the time separation function admits the Taylor expansion
\begin{align}
    l(\exp_x(w),\,\exp_{\gamma(\delta)}(\PT{\delta}\,w')) = \delta - \frac{\delta}{2}\,a_{ij}\,w^iw^j + \rmO(\delta\varepsilon^3) + \rmO(\delta^2\varepsilon^2).
\end{align}
\end{corollary}
\section{Unweighted reconstruction by weighted codimension one measures}\label{Ch:Det}

\subsection{Main result}\label{Sub:Framework} Recall $V\in C^\infty(\mms)$ is a fixed function inducing the weighted reference measure $\smash{\meas:=\rme^{-V}\,\vol_g}$. We continue to use the notation introduced in \cref{Sub:Fermi}. 

We define codimension one measures around $\gamma(t)$, where $t\in I$, as follows.
Consider the orthogonal complement $\smash{\dot\gamma(t)^\perp\subset T_{\gamma(t)}\mms}$ with respect to $g$, the restriction $h_t$ of $-g$ to the two-fold product $\smash{{\dot\gamma(t)^\perp}^2}$, the induced volume measure $\smash{\vol_{h_t}}$, and the induced $\varepsilon$-ball $\smash{B_\varepsilon^{h_t}(0)}\subset \dot \gamma(t)^\perp$, of dimension $n$, around zero. For $\varepsilon >0$ sufficiently small and $t$ sufficiently close to zero, $\smash{\exp_{\gamma(t)}}$ is a diffeomorphism on an open neighborhood of $\smash{B_\varepsilon^{h_t}(0)}$ in $T_{\gamma(t)}\mms$. Setting
\begin{align}\label{Eq:XiY}
    \xi_{\gamma(t)}^\varepsilon := \Big[\!\int_{B_\varepsilon^{h_t}(0)} \rme^{-V\circ\,\exp_{\gamma(t)}}\d\vol_{h_t}\Big]^{-1}\,\rme^{-V\circ\,\exp_{\gamma(t)}}\,\vol_{h_t}\mres B_\varepsilon^{h_t}(0),
\end{align}
which is a probability measure on $\smash{B_\varepsilon^{h_t}(0)}$,
\begin{align}\label{Eq:NuY}
    \nu_{\gamma(t)}^\varepsilon := (\exp_{\gamma(t)})_\push\xi_{\gamma(t)}^\varepsilon
\end{align}
is a probability measure on the spacelike hypersurface $\smash{\exp_{\gamma(t)}(B_\varepsilon^{h_t}(0))}$ through $\gamma(t)$.

\begin{figure}[t]
\centering
\begin{tikzpicture}[>=stealth, scale=1]
  
  \draw[gray] (0, 0) -- (3, 0.25) -- (3.5, 1.6) -- (0.5, 1.35) -- cycle;
  \node[gray, below left] at (0, 0) {$v^\perp$};
  
  \draw[gray] (1.4, 0.775) circle (0.45);
  \node[gray, right] at (1.85, 1) {$B_\varepsilon^{h_0}(0)$};
  
  \draw[gray] (0.5, 4.6) -- (3.5, 4.85) -- (4, 6.2) -- (1, 5.95) -- cycle;
  \node[gray, above right] at (4, 6.2) {$\dot\gamma(t)^\perp$};
  
  \draw[gray] (1.8, 5.375) circle (0.45);
  \node[gray, right] at (2.275, 5.55) {$B_\varepsilon^{h_t}(0)$};
  
  \draw[->, dashed, gray] (1.2, 1.65) .. controls (1, 3) .. (1.4, 4.5);
  \node[gray, left] at (1, 3.1) {$\PT{t}$};

  \draw (5.8, -0.3) .. controls (6.7,0.1) and (7.6,0.1) .. (8.5, 0);
  \node[right] at (8.55, 0) {$\supp\nu_x^\varepsilon$};
  
  \fill (7, 0.01) circle (1.5pt);
  \node[below] at (7, -0.1) {$x$};
  
  \draw[very thick] (7, 0.01) .. controls (6.6, 2) and (6.7, 4) .. (7.2, 5.3);
  \node at (7.1,3) {$\gamma$};
  
  \fill (7.2, 5.3) circle (1.5pt);
  \node[above right=-1pt] at (7.2, 5.3) {$\gamma(\delta)$};
  
  \draw (5.95, 6.3) .. controls (6.5,5.8) and (7.25,5.05) .. (8.5, 4.8);
  \node[right] at (8.55, 4.7) {$\supp \nu_{\gamma(t)}^\varepsilon$};

  \draw[->] (1.4, -0.1) .. controls (2, -0.8) and (4.0, -1.2) .. (5.4, -0.4);
  \node[below] at (4, -0.975) {$\exp_x$};
  
  \draw[->] (2, 4.55) .. controls (3.75, 3.3) and (5.2, 5) .. (6, 5.85);
  \node[below] at (3.7, 4) {$\exp_{\gamma(t)}$};
\end{tikzpicture}
\caption{Construction of the codimension-one slice measures \eqref{Eq:NuY}. The spacelike $\varepsilon$-balls $\smash{B_\varepsilon^{h_0}(0)\subset v^\perp}$ and $\smash{B_\varepsilon^{h_t}(0)\subset\dot\gamma(t)^\perp}$ are pushed forward by the exponential maps $\smash{\exp_x}$ and $\smash{\exp_{\gamma(t)}}$ (solid arrows) to the slice measures $\smash{\nu_x^\varepsilon}$ and $\smash{\nu_{\gamma(t)}^\varepsilon}$ on the corresponding spacelike hypersurfaces in $\mms$. The dashed arrow indicates parallel transport from $\smash{v^\perp}$ to $\smash{\dot\gamma(t)^\perp}$ along $\gamma$.}
\label{Fig:Slices}
\end{figure}
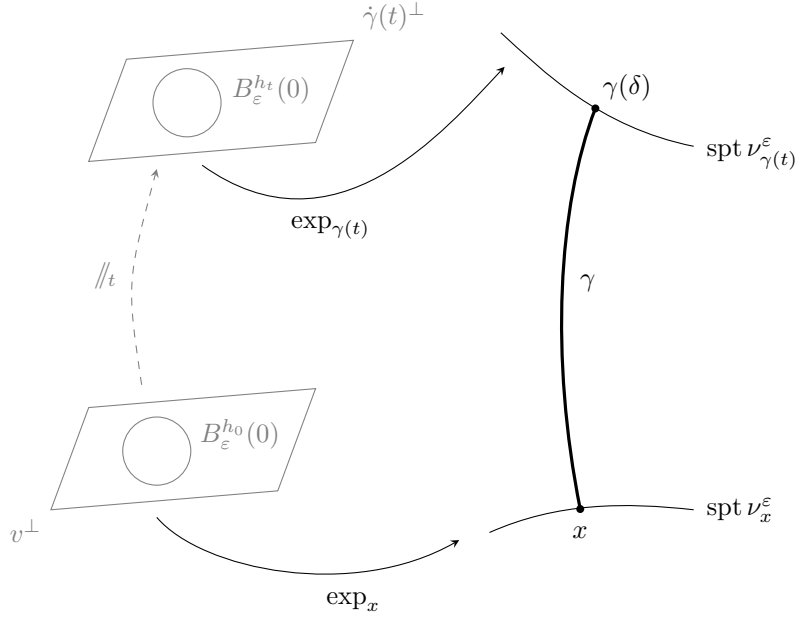

\begin{theorem}[Reconstruction of timelike Ricci curvature]\label{Th:MainInformal} Let $(\mms,g)$ designate a globally hyperbolic spacetime of dimension $n+1$, where $n\in\N$, and let $V\in C^\infty(\mms)$. Then for $\delta>0$ and $\varepsilon >0$ sufficiently small with $\smash{\varepsilon=\rmo(\delta^{5/2})}$, the Wasserstein 1-distances of the measures has the following asymptotic expansion
\begin{align}\label{no-V}
    \ell_1(\nu_x^\varepsilon,\nu_{\gamma(\delta)}^\varepsilon) = \delta\,\Big[1-\frac{\varepsilon^2}{2(n+2)}\,\Ric(v,v) + \rmO(\varepsilon^3)+\rmO(\delta\varepsilon^2)\Big].
\end{align}
\end{theorem}

Although the measures defined by (\ref{Eq:NuY}) involves $V$, the  Taylor expansion (\ref{no-V}) does not see $V$. This artefact of our codimension one construction is explained in \cref{Re:WeightInvisible}. On the other hand, when we ``thicken'' the above codimension one measures, the weight starts appearing in the expansion; cf.~\cref{Th:WeightedReconstruction}.

\subsection{Construction of a transport map on tangent spaces}\label{Sub:Transport} In view of establishing a lower bound on $\smash{\ell_1(\nu_x^\varepsilon,\nu_{\gamma(\delta)}^\varepsilon)}$ when $\delta>0$ and $\varepsilon>0$ are sufficiently small, we will now construct a map $\smash{T\colon \exp_x(B_\varepsilon^{h_0}(0))\to \exp_{\gamma(\delta)}(B_\varepsilon^{h_\delta}(0))}$ with $\smash{\nu_{\gamma(\delta)}^\varepsilon = T_\push\nu_x^\varepsilon}$. In the unweighted case, this is simple. Indeed, $\smash{\xi_x^\varepsilon}$ and $\smash{\xi_{\gamma(\delta)}^\varepsilon}$ then reduce to the uniform distributions on $\smash{B_\varepsilon^{h_0}(0)}$ and $\smash{B_\varepsilon^{h_\delta}(0)}$ with respect to $\smash{\vol_{h_0}}$ and $\smash{\vol_{h_\delta}}$, respectively. Since $\PT{\delta}$ is an isometry between the inner product spaces $\smash{(v^\perp,h_0)}$ and $\smash{(\dot\gamma(\delta)^\perp,h_\delta)}$,
\begin{align}\label{Eq:xixipushforward}
    \xi_{\gamma(\delta)}^\varepsilon = (\PT{\delta})_\push\xi_x^\varepsilon.
\end{align}
A natural map pushing $\smash{\nu_x^\varepsilon}$ forward to $\smash{\nu_{\gamma(\delta)}^\varepsilon}$ is thus
\begin{align}\label{Eq:Tdef}
    T := \exp_{\gamma(\delta)} \circ\, \PT{\delta}\circ \exp_{x}^{-1}.
\end{align}
In the weighted case, (\ref{Eq:xixipushforward}) fails in general, which necessitates an additional ``correction term'' in the definition (\ref{Eq:Tdef}) that takes the weight into account. It suffices to find a weighted modification of (\ref{Eq:xixipushforward}), i.e.~a map $\smash{S\colon B_{\varepsilon}^{h_0}(0)\to B_\varepsilon^{h_0}(0)}$ with
\begin{align}\label{Eq:Pushfor}
    \xi_{\gamma(\delta)}^\varepsilon = (\PT{\delta}\circ S)_\push\xi_x^\varepsilon.
\end{align}
Indeed, an admissible push-forward from $\smash{\nu_x^\varepsilon}$ to $\smash{\nu_{\gamma(\delta)}^\varepsilon}$ is then given by
\begin{align}\label{Eq:TdefWeighted}
    T:=\exp_{\gamma(\delta)}\circ\,\PT{\delta}\circ S \circ\exp_x^{-1}.
\end{align}

Observe $\smash{(\PT{\delta}^{-1})_\push\xi_{\gamma(\delta)}^\varepsilon}$ is a probability measure on $\smash{B_\varepsilon^{h_0}(0)}$. Hence, it suffices to construct a map pushing forward $\smash{\xi_x^\varepsilon}$ to  $\smash{(\PT{\delta}^{-1})_\push\xi_{\gamma(\delta)}^\varepsilon}$ with the last two claimed properties. 

The main aim of the Lemma is obtain a quantitative error. The proof follows from a quantitative Moser gradient  argument \cite{DacorognaMoser1990,Moser1965} --- which we reproduce below for subsequent use, and Schauder estimates.

\begin{lemma}[Correction map]\label{Le:CorrectionMap}
Under the assumptions of \cref{Re:SmallEnough}, there exists
a diffeomorphism
$S\colon B_\varepsilon^{h_0}(0)\to B_\varepsilon^{h_0}(0)$
satisfying \eqref{Eq:Pushfor}, extending smoothly to the closure
with
$S(\partial B_\varepsilon^{h_0}(0))
=\partial B_\varepsilon^{h_0}(0)$, and obeying
\begin{align}\label{Eq:ExpansionS}
    S=\Id+\rmO(\delta\varepsilon^2),
\end{align}
uniformly on $B_\varepsilon^{h_0}(0)$.
\end{lemma}

\begin{proof} 
 Let $\mu_0=\xi_x^\epsilon$ and $\mu_1=\smash{(\PT{\delta}^{-1})_\push\xi_{\gamma(\delta)}^\varepsilon}$, suppressing the parameters. We abbreviate the densities of the measures  with respect to $\smash{\vol_{h_0}}$ by $\rho_0$ and $\rho_1$, respectively. Given $t\in[0,1]$, define the interpolation density $\smash{\rho_t\colon v^\perp\to\R_+}$ by
 $$\rho_t=(1-t)\rho_0+t\rho_1.$$
Write $|B_\varepsilon^{h_0}(0)|=\vol_{h_0}[B_\varepsilon^{h_0}(0) ]$ for the $h_0$-volume
of the ball.
 Since $V$ is smooth, $\rho_t$ is smooth on $\smash{B_\varepsilon^{h_0}(0)}$. Moreover, since $\smash{\rme^{-V}}$ is locally bounded  above and away from zero on the chosen neighbourhood, there exist $\lambda>0$ and $\Lambda>0$ such that when $\varepsilon>0$ and $\delta>0$ are sufficiently small, every $t\in[0,1]$  obeys
 \begin{equation}\label{Eq:DensEst}
    \lambda\,|B_\varepsilon^{h_0}(0)|^{-1}
    1_{B_\varepsilon^{h_0}(0)}
    \leq\rho_t\leq
    \Lambda\,|B_\varepsilon^{h_0}(0)|^{-1}
    1_{B_\varepsilon^{h_0}(0)}.
\end{equation}

 Given $t\in [0,1]$, let  $\psi_t$ denote the solution to the Poisson equation on $B_\epsilon^{h_0}(0)$:
\begin{equation}\label{Eq:EllPDE}
\div(\rho_t \nabla^{h_0}\psi_t)=\rho_0-\rho_1
\end{equation}
with the no-flux condition $\<\nabla^{h_0} \psi_t, \vec{n}\>_{h_0}=0$ where $\vec{n}$ denote the normal vector. 

By \eqref{Eq:DensEst}, the operator is uniformly elliptic.
Since $\rho_0$ and $\rho_1$ are probability densities,
the right-hand side has zero integral. The Neumann problem
therefore admits a solution, unique up to an additive
constant, which we fix by
\begin{align}\label{Eq:Meanzeroconstraint}
    \int_{B_\varepsilon^{h_0}(0)}
    \psi_t\,\d\vol_{h_0}=0.
\end{align}

Set $v_t= \nabla^{h_0}\psi_t$. Since $\partial_t\rho_t=\rho_1-\rho_0$, the PDE rewrites as $\partial_t\rho_t+\div(\rho_tv_t)=0$.
Let $\smash{\Phi \colon [0,1]\times B_\varepsilon^{h_0}(0)\to v^\perp}$ denote the smooth gradient flow of $\nabla^{h_0}\psi_t$, i.e.~for every $\smash{w\in B_\varepsilon^{h_0}(0)}$ it solves the ODE driven by $\psi_t$: 
$$\label{Eq:IVP}
\Phi_t(w)=w+ \int_0^t \nabla^{h_0} \psi_s(\Phi_s(w))\; ds.$$
Since 
$\left\langle\nabla^{h_0}\psi_t,n\right\rangle_{h_0}=0$
on $\partial B_\varepsilon^{h_0}(0)$, the vector field $v_t=\nabla^{h_0}\psi_t$ is tangent to the boundary. Hence $\Phi_t$ leaves the boundary invariant and is a diffeomorphism of $\smash{\overline{B}_\varepsilon^{h_0}(0)}$ onto itself.

Set $S=\Phi_1$ which depends on $\varepsilon$ and $\delta$. The fact that $S$ is the transport map follows from the continuity equation. Let $J_t$ denote the scalar Jacobian density relative to $\vol_{h_0}$, defined by
$(\Phi_t)_*\vol_{h_0}=J_t\,\vol_{h_0}$. 
Then
$$
(\Phi_t)_*\mu_0
=J_t \, \rho_0\circ\Phi_t^{-1}\,d\vol_{h_0}.
$$
For every $f\in C^\infty(\overline{B_\varepsilon})$,
$$\begin{aligned}
\frac{\rmd}{\rmd t}\int_{B_\epsilon} f\circ \Phi_t(x)\d\mu_0(x)
&=\int_{B_\epsilon}\langle \nabla^{h_0} f,v_t\rangle_{h_0}\,J_t\,\rho_0 \circ \Phi_t^{-1} \d\vol_{h_0}\\
&= -\int_{B_\epsilon}f \div (v_t\,J_t \, \rho_0 \circ \Phi_t^{-1}) \d\vol_{h_0}.
\end{aligned}
$$
Since $v_t$ is tangent to the boundary, the boundary term vanishes. It follows that
$\partial_t (J_t\rho_0\circ \Phi_t^{-1}) +\div(v_t J_t\rho_0\circ \Phi_t^{-1})=0$.
Since $\rho_t$ solves the same continuity equation with the same initial value, uniqueness gives
$J_t\, \rho_0\circ\Phi_t^{-1}=\rho_t$ and consequently $S_*\mu_0=(\Phi_1)_*\mu_0=\mu_1$.

  It remains to show the correction term for $S$ is of order $ O(\delta \epsilon^2)$. To this end it is sufficient to prove the following gradient estimate, uniformly in $t\in[0,1]$,
\begin{equation}   \label{uniform-gradient}
 \big\Vert \nabla^{h_0}\psi_t\big\Vert_{C^0(B_\varepsilon^{h_0}(0);Tv^\perp)}=\rmO(\delta\varepsilon^2).
\end{equation}
which we obtain below with Sobolev estimates by first gathering the equations to the same domain $B_1^{h_0}$. 
\cref{uniform-gradient} combines with (\ref{Eq:IVP})  implies the desired asymptotic expansion
$$
    \Phi_1 = \Id + \int_{[0,1]} \nabla^{h_0}\psi_t\circ\Phi_t\d t = \Id +\rmO(\delta\varepsilon^2)$$

To prove \eqref{uniform-gradient}, set
\begin{align*}
    G_r(w):=\rme^{V(\gamma(r))
    -V(\exp_{\gamma(r)}(\PT{r}\,w))},\quad 
    Z_r:=\int_{B_\varepsilon^{h_0}(0)}
    G_r(w)\,\d\vol_{h_0}(w).
\end{align*}
Then $\rho_0=G_0/Z_0$ and $\rho_1=G_\delta/Z_\delta$.
Since $G_r(0)=1$ for every $r$, we have
$\partial_rG_r(0)=0$. Smoothness therefore gives
\begin{align*}
    |\partial_rG_r(w)|\leq C|w|_{h_0},
    \qquad
    |\partial_rZ_r|
    \leq C\varepsilon|B_\varepsilon^{h_0}(0)|.
\end{align*}
Moreover, the Taylor expansion
\begin{align*}
    G_r(w)
    =1-\rmd V_{\gamma(r)}(\PT{r}\,w)
    +\rmO(|w|_{h_0}^2)
\end{align*}
and symmetry of the ball imply
\begin{align*}
    Z_r=|B_\varepsilon^{h_0}(0)|
    \bigl[1+\rmO(\varepsilon^2)\bigr].
\end{align*}
Differentiating the normalized density thus yields
\begin{align*}
    \partial_r\left(\frac{G_r}{Z_r}\right)
    &=\frac{\partial_rG_r}{Z_r}
    -\frac{G_r\,\partial_rZ_r}{Z_r^2}
    =|B_\varepsilon^{h_0}(0)|^{-1}\rmO(\varepsilon).
\end{align*}
Integrating from $r=0$ to $r=\delta$, we obtain
\begin{align}\label{Eq:rho01diff}
    \rho_0(w)-\rho_1(w)
    =|B_\varepsilon^{h_0}(0)|^{-1}
    \rmO(\delta\varepsilon),
\end{align}
uniformly in $w\in B_\varepsilon^{h_0}(0)$. 
We now use this identity to show (\ref{Eq:ExpansionS}). 

For $t\in[0,1]$ and $z\in B_1^{h_0}(0)$, set
\begin{align*}
    \tilde\rho_t(z)
    &:=|B_\varepsilon^{h_0}(0)|\,\rho_t(\varepsilon z),
    &\tilde\psi_t(z)&:=\psi_t(\varepsilon z).
\end{align*}
By \eqref{Eq:DensEst}, \eqref{Eq:rho01diff}, and smoothness
of $G_r$, we have, uniformly in $t$,
\begin{align}\label{Eq:lambdalambda}
    \lambda\leq\tilde\rho_t\leq\Lambda,
    \qquad
    \|\tilde\rho_t\|_{C^1(\overline B_1)}\leq C,
    \qquad
    \|\tilde\rho_0-\tilde\rho_1\|_{C^0(\overline B_1)}
    \leq C\delta\varepsilon.
\end{align}
Here $B_1:=B_1^{h_0}(0)$, and the derivative bound follows
from
\begin{align*}
    \nabla_z\left(
    \frac{|B_\varepsilon^{h_0}(0)|G_r(\varepsilon z)}{Z_r}
    \right)
    =
    \frac{\varepsilon|B_\varepsilon^{h_0}(0)|}{Z_r}
    (\nabla_wG_r)(\varepsilon z).
\end{align*}

Rescaling \eqref{Eq:EllPDE} gives
\begin{align*}
    \div_{h_0}
    (\tilde\rho_t\nabla^{h_0}\tilde\psi_t)
    &=\varepsilon^2(\tilde\rho_0-\tilde\rho_1)
    &&\text{in }B_1^{h_0}(0)\subset v^\perp,\\
    \partial_\nu\tilde\psi_t&=0
    &&\text{on }\partial B_1^{h_0}(0).
\end{align*}
The mean-zero condition is preserved by rescaling.
The Neumann $W^{2,p}$ estimate, followed by Sobolev embedding
for $p>\dim v^\perp$ (cf.~\cite{tartar2007sobolev}) gives
\begin{align*}
    \|\nabla^{h_0}\tilde\psi_t\|_{C^0(\overline B_1)}
    &\leq C\|\tilde\psi_t\|_{W^{2,p}(B_1)}\\
    &\leq C\varepsilon^2
    \|\tilde\rho_0-\tilde\rho_1\|_{L^p(B_1)}
    \leq C\delta\varepsilon^3.
\end{align*}
The constant is uniform by \eqref{Eq:lambdalambda}.
Rescaling back yields
\begin{align*}
    \|\nabla^{h_0}\psi_t\|_{C^0(
    \overline{B_\varepsilon^{h_0}(0)})}
    =\varepsilon^{-1}
    \|\nabla^{h_0}\tilde\psi_t\|_{C^0(\overline B_1)}
    \leq C\delta\varepsilon^2,
\end{align*}
uniformly in $t\in[0,1]$. This proves \eqref{uniform-gradient}. Finally, since $\Phi_t(w)\in B_\varepsilon^{h_0}(0)$,
the flow equation and \eqref{uniform-gradient} give
$$
    |S(w)-w|_{h_0}
   \leq\int_0^1	  |\nabla^{h_0}\psi_t(\Phi_t(w))|_{h_0}\,\d t
  \leq C\delta\varepsilon^2,
$$
uniformly in $w\in B_\varepsilon^{h_0}(0)$.
This proves \eqref{Eq:ExpansionS} and completes the proof.
\end{proof}

We will also need to compute the second spatial moments of $\smash{\xi_x^\varepsilon}$, which naturally appear through the time separation expansion \cref{Le:LengthExpansion,Le:UpperSep}; this is where the Ricci curvature shows up via \cref{Co:RicciBall}. Recall the matrix $a$ from \cref{Sub:Fermi}.

\begin{lemma}[Second spatial moments of \eqref{Eq:XiY}]\label{Le:DensityExpansion} For $\varepsilon>0$ sufficiently small, 
\begin{align}\label{Eq:weightedmoment}
a_{ij}\int_{B_\varepsilon^{h_0}(0)} w^iw^j\d\xi_x^\varepsilon(w) = \frac{\varepsilon^2}{n+2}\,\Ric(v,v) + \rmO(\varepsilon^4).
\end{align}
\end{lemma}

\begin{proof} Following the computations from the previous proof, we see the Radon--Nikodým density $\rho_0$ of $\smash{\xi_x^\varepsilon}$ with respect to $\vol_{h_0}$ obeys, uniformly in $\smash{w\in B_\varepsilon^{h_0}(0)}$,
\begin{align}\label{Eq:rho0exp}
\rho_0(w) = \vol_{h_0}\big[B_\varepsilon^{h_0}(0)\big]^{-1}\,\big[1 - \rmd V(w) + \rmO(\varepsilon^2)\big].
\end{align}
We multiply this with $\smash{a_{ij}\,w^iw^j}$ and integrate with respect to $\smash{\vol_{h_0}}$. This entails the asymptotic expansion
\begin{align*}
    a_{ij}\int_{B_\varepsilon^{h_0}(0)}w^iw^j\d\xi_x^\varepsilon(w)
   & =    a_{ij}\int_{B_\varepsilon^{h_0}(0)}w^iw^j\,\rho_0(w)\d \vol_{h_0}(w)\\
    &\qquad= a_{ij}\,\vol\big[B_\varepsilon^{h_0}(0)\big]^{-1}\int_{B_\varepsilon^{h_0}(0)}w^iw^j\d\vol_{h_0}(w)\\
    &\qquad\qquad - a_{ij}\,\vol\big[B_\varepsilon^{h_0}(0)\big]^{-1}\int_{B_\varepsilon^{h_0}(0)}w^iw^j\d V(w)\d\vol_{h_0}(w)\\
    &\qquad\qquad + a_{ij}\,\vol\big[B_\varepsilon^{h_0}(0)\big]^{-1}\int_{B_\varepsilon^{h_0}(0)}w^iw^j\,\rmO(\varepsilon^2) \d\vol_{h_0}(w).
\end{align*}
We treat each contribution separately.

For the first summand, we invoke \cref{Co:RicciBall} to infer
\begin{align*}
    a_{ij}\,\vol\big[B_\varepsilon^{h_0}(0)\big]^{-1}\int_{B_\varepsilon^{h_0}(0)}w^iw^j\d\vol_{h_0}(w) =  \frac{\varepsilon^2}{n+2}\,\Ric(v,v).
\end{align*}

The second summand contains an integral of an antisymmetric integrand on a symmetric domain. Thus, it vanishes identically.

Lastly, since $a_{ij}=\rmO(1)$ as in the proof of \cref{Le:LengthVariation} while $\smash{w^iw^j =\rmO(\varepsilon^2)}$ for every $i,j\in\{1,\dots,n\}$, the contribution from the remainder terms is in $\rmO(\varepsilon^4)$.
\end{proof}

\begin{remark}[Why the weight is invisible at this order]\label{Re:WeightInvisible} The first-order term $\rmd V(w)$ from (\ref{Eq:rho0exp}) does \emph{not} contribute to the spatial second moment (\ref{Eq:weightedmoment}), by oddness. This is the codimension one analog of the well-known Riemannian phenomenon that the gradient of the potential only enters the \emph{first}-order coefficient of the length expansion, cf.~Arnaudon--Li--Petko \cite{arnaudon-li-petko2025}*{Lem.~2.2}. In our spacelike construction (\ref{Eq:NuY}) of the transported measures, the geodesic from $x$ to $\gamma(\delta)$ is orthogonal to both slices, so there is no first-order length contribution to carry the potential gradient; the weight reappears only through the temporal smearing of \cref{Th:WeightedReconstruction}, where the Hessian of $V$ enters via the second moment in the timelike direction.
\end{remark}

\subsection{Lower bound of \cref{Th:MainInformal}}\label{Sub:Lower} In the next two subsections, we will prove \cref{Th:MainInformal} by establishing a lower and an upper bound on $\smash{\ell_1(\nu_x^\varepsilon,\nu_{\gamma(\delta)}^\varepsilon)}$ for appropriately chosen $\delta >0$ and $\varepsilon>0$. We point out the similarity to \cref{Le:LengthExpansion,Le:UpperSep} implying \cref{Co:TimeSepExpansion}. The lower bound we address now is obtained by exhibiting one admissible chronological coupling --- the one induced by the approximate transport map $T$ built from (\ref{Eq:TdefWeighted})--- and estimating the average time separation it realizes. As $\ell_1$ is a supremum, this choice yields a lower bound.

\begin{proposition}[Asymptotic lower bound on Lorentz--Wasserstein distance]\label{Pr:LowerBound} For $\delta>0$ and $\varepsilon>0$ sufficiently small with $\varepsilon = \rmo(\delta)$,
\begin{align}
\ell_1(\nu_x^\varepsilon,\nu_{\gamma(\delta)}^\varepsilon) \geq \delta\,\Big[1 - \frac{\varepsilon^2}{2(n+2)}\,\Ric(v,v) + \rmO(\varepsilon^3) + \rmO(\delta\varepsilon^2)\Big].
\end{align}
\end{proposition}

\begin{proof} 
Let $\pi$ denote the coupling $\smash{\nu_x^\varepsilon}$ and $\smash{\nu_{\gamma(\delta)}^\varepsilon}$ induced by the the transport map $T$ from  \cref{Le:CorrectionMap}.
$$
\pi := (\Id,T)_\push\nu_x^\varepsilon
$$
It concentrated on the graph of $T$. It is also chronological since all slice points in question are chronologically related within the fixed normal  neighborhood. By the definition of $T$ and \eqref{Eq:NuY},
\begin{align}\label{Eq:LowerStart}
\begin{split}
\ell_1(\nu_x^\varepsilon,\nu_{\gamma(\delta)}^\varepsilon) &\geq \int_\mms l\circ(\Id,T)\d\nu_x^\varepsilon\\
&= \int_{B_\varepsilon^{h_0}(0)} l(\exp_x(w),\exp_{\gamma(\delta)}(\PT{\delta}\,S(w)))\d\xi_x^\varepsilon(w).
\end{split}
\end{align}

We now estimate the integrand from below. \cref{Le:CorrectionMap} gives 
\begin{align*}
    \big\vert S(w)-w\big\vert_{h_0}=\rmO(\delta\varepsilon^2)
\end{align*}
uniformly in $\smash{w\in B_\varepsilon^{h_0}(0)}$. In particular, the hypothesis of \cref{Le:LengthExpansion} is met with $w'=S(w)$, which yields the following uniformly in $\smash{w\in B_\varepsilon^{h_0}(0)}$:
$$
l(\exp_x(w),\exp_{\gamma(\delta)}(\PT{\delta}\,S(w))) \geq \delta - \frac\delta2\,a_{ij}\,w^iw^j + \rmO(\delta\varepsilon^3) + \rmO(\delta^2\varepsilon^2).
$$
Inserting this into \eqref{Eq:LowerStart}, using $\smash{\xi_x^\varepsilon}$ is a probability measure concentrated on $\smash{B_\varepsilon^{h_0}(0)}$, and finally applying \cref{Le:DensityExpansion},
\begin{align*}
\ell_1(\nu_x^\varepsilon,\nu_{\gamma(\delta)}^\varepsilon) &\geq \delta - \frac\delta2\,a_{ij}\int_{B_\varepsilon^{h_0}(0)} w^iw^j\d\xi_x^\varepsilon(w) + \rmO(\delta\varepsilon^3) + \rmO(\delta^2\varepsilon^2)\\
&= \delta - \frac{\delta\varepsilon^2}{2(n+2)}\,\Ric(v,v) + \rmO(\delta\varepsilon^3)+\rmO(\delta^2\varepsilon^2);
\end{align*}
here, we have absorbed the additional contribution $\smash{\rmO(\delta\varepsilon^4)}$ into $\smash{\rmO(\delta\varepsilon^3)}$.
\end{proof}

\subsection{Upper bound of \cref{Th:MainInformal}}\label{Sub:Upper} The upper bound requires us to control an \emph{arbitrary} admissible coupling in view of \cref{Def:LW}. The punchline is that \cref{Le:UpperSep} reduces the Taylor expansion of the time separation function to the first marginal of such a coupling, to which \cref{Le:DensityExpansion} applies.

\begin{proposition}[Asymptotic upper bound on Lorentz--Wasserstein distance]\label{Pr:UpperBound} If we strengthen the hypotheses of \cref{Pr:LowerBound} to $\varepsilon = \rmo(\delta^{5/2})$,
\begin{align}
\ell_1(\nu_x^\varepsilon,\nu_{\gamma(\delta)}^\varepsilon) \leq \delta\,\Big[1 - \frac{\varepsilon^2}{2(n+2)}\,\Ric(v,v) + \rmO(\varepsilon^3) + \rmO(\delta\varepsilon^2)\Big].
\end{align}
\end{proposition}

\begin{proof} Since the chronological relation $\smash{I_g}$ is open, if $\delta$ and $\varepsilon$ are sufficiently small every point in the support of $\smash{\nu_x^\varepsilon}$ lies in the chronological past of every point in the support of $\smash{\nu_{\gamma(\delta)}^\varepsilon}$. Let $\pi$ be an arbitrary coupling of $\smash{\nu_x^\varepsilon}$ and $\smash{\nu_{\gamma(\delta)}^\varepsilon}$, which is necessarily chronological by the previous clause. Thanks to \eqref{Eq:NuY},
\begin{align*}
    \varpi := (\exp_x^{-1},\PT{\delta}^{-1}\circ\exp_{\gamma(\delta)}^{-1})_\push\pi
\end{align*}
is a coupling of $\smash{\xi_x^\varepsilon}$ and $\smash{(\PT{\delta}^{-1})_\push\xi_{\gamma(\delta)}^\varepsilon}$; by construction, it obeys
\begin{align}\label{Eq:lcombines}
    \int_{\mms^2}l\d\pi = \int_{B_\varepsilon^{h_0}(0)^2} l(\exp_x(w),\exp_{\gamma(\delta)}(\PT{\delta}\,w')) \d\varpi(w,w').
\end{align}
By \cref{Le:UpperSep}, uniformly in $\smash{w,w'\in B_\varepsilon^{h_0}(0)}$,
\begin{align*}
l(\exp_x(w),\exp_{\gamma(\delta)}(\PT{\delta}\,w')) &\leq\delta- \frac\delta2\,a_{ij}\,w^iw^j + \rmO(\delta\varepsilon^3) + \rmO(\delta^2\varepsilon^2),
\end{align*}
which combines with (\ref{Eq:lcombines}) to yield
\begin{align*}
    \int_{\mms^2}l\d\pi \leq \delta - \frac{\delta}{2}\,a_{ij}\int_{B_\varepsilon^{h_0}(0)^2} w^iw^j\d\varpi(w,w') + \rmO(\delta\varepsilon^3)+\rmO(\delta^2\varepsilon^2),
\end{align*}
where here and in the sequel, the Lagrange remainders do not depend on $\pi$ since it integrates up to one. Since the integral on the right-hand side integrates out the second component, using \cref{Le:DensityExpansion} we arrive at
\begin{align*}
    -\frac{\delta}{2}\,a_{ij}\int_{B_\varepsilon^{h_0}(0)^2} w^iw^j\d\varpi(w,w') &= -\frac{\delta}{2}\,a_{ij}\int_{B_\varepsilon^{h_0}(0)} w^iw^j\d\xi_x^\varepsilon(w)\\
    &= -\frac{\delta\varepsilon^2}{2(n+2)}\,\Ric(v,v) + \rmO(\delta\varepsilon^4).
\end{align*}
The contribution from $\rmO(\delta\varepsilon^4)$ is absorbed by $\smash{\rmO(\delta\varepsilon^3)}$. The arbitrariness of $\pi$ thus establishes the desired upper bound.
\end{proof}

Combining \cref{Pr:LowerBound,Pr:UpperBound} readily gives \cref{Th:MainInformal}.

\begin{remark}[The unweighted case]\label{Re:Unweighted} Provided $V$ vanishes identically, the induced reference measure $\meas$ is the volume measure $\vol_g$, the correction map of \cref{Sub:Transport} reduces to the identity by (\ref{Eq:xixipushforward}), and $\rho_0$ from \cref{Le:DensityExpansion} is the uniform density on $\smash{B_\varepsilon^{h_0}(0)}$. The proof simplifies accordingly: the transport map $T$ is (\ref{Eq:Tdef}) and \cref{Le:DensityExpansion} is unnecessary, the relevant second moment being supplied directly by \cref{Co:RicciBall}. The expansion from \cref{Th:MainInformal} holds as stated.
\end{remark}

\section{Weighted reconstruction of Bakry--Émery curvature}\label{Ch:Weighted}

We now recover the weight $V\in C^\infty(\mms)$, which induces the reference measure $\smash{\meas:= \rme^{-V}\,\vol_g}$, through second order effects in the temporal direction.
To this end, we smear the spacelike slice measures of the previous section into thin tubes of temporal height $2\eta$, where $\eta>0$.
The calibration $\eta^2=\frac{3\varepsilon^2}{2(n+2)}$ between the temporal and spatial scales, cf.~\eqref{Eq:PrecCal}, yields the {damped (Bakry--\'Emery) Ricci tensor} $\Ric+\Hess V$; see \cref{Th:WeightedReconstruction}. This is consistent with the weighted Riemannian reconstruction of Arnaudon--Li--Petko \cite{arnaudon-li-petko2025} for {uniform probability measures on balls}. The strategy largely follows that of the previous section, sp some proofs will only be sketched emphasising the main differences.

\subsection{Thin tube measures}\label{Sub:TubeMeas} 
Let $v\in T_x\mms$ be a future-directed unit timelike vector, and let $\gamma(t):=\exp_x(tv)$, $t\in I$,
be the locally defined geodesic parametrised by proper time.
Write
\begin{align*}
    v^\perp:=\{w\in T_x\mms:g(v,w)=0\}.
\end{align*}
For sufficiently small $\varepsilon,\eta>0$, define the cylinder
\begin{align*}
    D_{\varepsilon,\eta}
    :=(-\eta,\eta)\times B_\varepsilon^{h_0}(0)
    \subset I\times v^\perp
\end{align*}
and the map $\Psi_\gamma\colon I\times D_{\varepsilon,\eta}\to\mms$
by
\begin{align}\label{Eq:psidef}
    \Psi_\gamma(t,s,w)
    :=\Psi_{\gamma(t)}(s,w)
    :=\exp_{\gamma(t)}
    \bigl(s\dot\gamma(t)+\PT{t}\,w\bigr).
\end{align}
After shrinking $I$ and taking $\varepsilon,\eta$ sufficiently small,
we may assume that $\Psi_\gamma$ is smooth, takes values in the
tubular neighbourhood $U$ of $\gamma$ from \cref{Sub:Fermi},
and that each $\Psi_{\gamma(t)}$ is a diffeomorphism onto its image.
Following the notation of \cref{Sub:Fermi}, for each $t\in I$, $\Psi_{\gamma(t)}$ is an exponential chart centred at $\gamma(t)$, up to an isometric identification of $(v^\perp,h_0)$ with $(\R^n,\hEucl)$.

For $t\in I$, define the timelike tube centred at $\gamma(t)$, of spatial thickness $\varepsilon$ and temporal height
$2\eta$, by
\begin{align*}
    \mathscr{T}_{\gamma(t)}^{\varepsilon,\eta}
    :=\Psi_{\gamma(t)}(D_{\varepsilon,\eta})
    \subset \mms.
\end{align*}
The associated weighted uniform probability measure is
\begin{align}\label{Eq:mux}
    \mu_{\gamma(t)}^{\varepsilon,\eta}
    :=\frac{
        \meas\mres\mathscr{T}_{\gamma(t)}^{\varepsilon,\eta}
    }{
        \meas\bigl[\mathscr{T}_{\gamma(t)}^{\varepsilon,\eta}\bigr]
    },
    \qquad
    \meas:=\rme^{-V}\,\vol_g.
\end{align}

As \cref{Sub:Framework}, we write this measure as the push-forward
of a probability measure on $D_{\varepsilon,\eta}$.
To this end, let
$J_{\gamma(t)}\colon D_{\varepsilon,\eta}\to (0,\infty)$ denote the volume Jacobian of $\Psi_{\gamma(t)}$ relative to $\Leb^1\otimes\vol_{h_0}$:
\begin{align}\label{Eq:Jacobian}
    J_{\gamma(t)}(s,w)
    := \sqrt{
            \left|\det\bigl((\Psi_{\gamma(t)}^*g)_{(s,w)}\bigr)\right|
        }= \sqrt{\big\vert\!\det g\circ\Psi_{\gamma(t)}(s,w)\big\vert},
\end{align}
where the determinant in the last expression is taken in the exponential coordinates associated with $\Psi_{\gamma(t)}$ and an $h_0$-orthonormal basis of $v^\perp$.
Define the probability measure on $D_{\varepsilon,\eta}$ by
\begin{align}\label{Eq:chix}
    \chi_{\gamma(t)}^{\varepsilon,\eta}
    :=       \frac{
            \rme^{-V\circ\Psi_{\gamma(t)}}\,J_{\gamma(t)}
            \,(\Leb^1\otimes\vol_{h_0})\mres D_{\varepsilon,\eta}
        }{
            \displaystyle\int_{D_{\varepsilon,\eta}}
            \rme^{-V\circ\Psi_{\gamma(t)}}\,J_{\gamma(t)}
            \,\d(\Leb^1\otimes\vol_{h_0})
        }.
\end{align}
The change-of-variables formula gives
\begin{align}\label{Eq:muxpushforward}
    \mu_{\gamma(t)}^{\varepsilon,\eta}
    =(\Psi_{\gamma(t)})_\push
    \chi_{\gamma(t)}^{\varepsilon,\eta}.
\end{align}

Alternatively, omitting the Jacobian from \eqref{Eq:chix}
gives the simpler probability measure
\begin{align}\label{Eq:tildechi}
    \widetilde{\chi}_{\gamma(t)}^{\varepsilon,\eta}
    :=
        \frac{
            \rme^{-V\circ\Psi_{\gamma(t)}}
            \,(\Leb^1\otimes\vol_{h_0})\mres D_{\varepsilon,\eta}
        }{
            \displaystyle\int_{D_{\varepsilon,\eta}}
            \rme^{-V\circ\Psi_{\gamma(t)}}
            \,\d(\Leb^1\otimes\vol_{h_0})
    }.
\end{align}
Its push-forward under $\Psi_{\gamma(t)}$ is generally no longer
the uniform probability measure on
$\mathscr{T}_{\gamma(t)}^{\varepsilon,\eta}$ with respect to $\meas$, making this choice less natural.
Nevertheless, omitting the Jacobian does not affect the
asymptotic expansion of the $1$-Lorentz--Wasserstein distance
at the order considered here. Taylor's formula therefore gives $J_{\gamma(t)}(s,w)=1+\rmO(s^2)+\rmO(|w|_{h_0}^2)$, so the Jacobian contributes only to the quadratic remainder in the expansion of the unnormalized weighted density, see the proof of Lemma~\ref{Le:WeightedDensityExpansion}.

\subsection{Moments of thin tube measures} 
The proof of \cref{Th:WeightedReconstruction} requires computing the moments of the measure
$\chi_{\gamma(t)}^{\varepsilon,\eta}$ that appear in the time separation expansion in \cref{Le:TubeTimeSep}.
The following result should be compared with the slice version
recorded in \cref{Le:DensityExpansion}. Recall from \cref{Sub:Fermi} the expansion
of the metric coefficient $g_{00}$:
$$
g_{00}(t,w)
=1-( a_{ij}+\rmO(t))\,w^iw^j
+\rmO\bigl(|w|_{h_0}^3\bigr).
$$

\begin{lemma}[Weighted density lemma]
\label{Le:WeightedDensityExpansion}
Write
$|D_{\varepsilon,\eta}|
:=2\eta\,\vol_{h_0}[B_\varepsilon^{h_0}(0)]$,
and let $\rho_t^{\varepsilon,\eta}$ denote the density of
$\chi_{\gamma(t)}^{\varepsilon,\eta}$ with respect to
$\Leb^1\otimes\vol_{h_0}$.
After shrinking $I$, the following expansion holds uniformly
in $t\in I$ and $(s,w)$ sufficiently close to zero:
\begin{align*}
    \rme^{-V\circ\Psi_{\gamma(t)}(s,w)}J_{\gamma(t)}(s,w)
    &=\rme^{-V(\gamma(t))}
    \bigl[1-s\rmd V(\dot\gamma(t))
    -\rmd V(\PT{t}\,w)+R_t(s,w)\bigr],
\end{align*}
where the smooth remainder satisfies
\begin{align}\label{error-1}
    |R_t(s,w)|+|\partial_tR_t(s,w)|
    \leq C\bigl(s^2+|w|_{h_0}^2\bigr).
\end{align}
Consequently, for sufficiently small $\varepsilon,\eta>0$,
\begin{align}\label{density}
    \meas\big[\mathscr{T}_{\gamma(t)}^{\varepsilon,\eta}\big]
    &=\rme^{-V(\gamma(t))}|D_{\varepsilon,\eta}|
    \bigl[1+\rmO(\eta^2)+\rmO(\varepsilon^2)\bigr],
\end{align}
and, uniformly in $t,u\in I$ and
$(s,w)\in D_{\varepsilon,\eta}$,
\begin{align*}\
&\f{d\chi_{\gamma(t)}^{\varepsilon,\eta}}{d(\Leb^1\otimes\vol_{h_0})}
=\f{ \bigl[1-s\rmd V(\dot\gamma(t))
    -\rmd V(\PT{t}\,w)+R_t(s,w)\bigr]} {|D_{\varepsilon,\eta}|
    \bigl[1+\rmO(\eta^2)+\rmO(\varepsilon^2)\bigr]},
\\
   & |D_{\varepsilon,\eta}|\,
    \bigl|\rho_t^{\varepsilon,\eta}(s,w)
    -\rho_u^{\varepsilon,\eta}(s,w)\bigr|
    \leq C|t-u|(\eta+\varepsilon).
\end{align*}
\end{lemma}

\begin{proof}
Recall that $\Psi_{\gamma(t)}$ defines normal coordinates,  in particular $\Psi_{\gamma(t)}(0,0)=\gamma(t)$, $(D\Psi_{\gamma(t)})_{(0,0)}(s,w)=s\dot \gamma(t)+\PT{t}w$.
 Set
\begin{align*}
    F_t(s,w)
    :=\rme^{V(\gamma(t))-V\circ\Psi_{\gamma(t)}(s,w)}
    J_{\gamma(t)}(s,w).
\end{align*}
Since $J=\sqrt{|\det g|}$, $\partial_iJ_p(x)=\f 12 J_p(x)g^{jk}(x)\partial_ig_{jk}(x)$, $J_{\gamma(t)}(0,0)=1$ and its first derivatives in $(s,w)$ vanish at the origin,  we have $F_t(0,0)=1$ and
\begin{align*}
    (D_{(s,w)}F_t)_{(0,0)} (s,w):=\f d {dr}F_t(rs, rw)|_{r=0}
=-s\rmd V(\dot\gamma(t))-\rmd V(\PT{t}\,w),
\end{align*}
Taylor's formula gives the asserted expansion $$F_t(s,w)=1-s\rmd V(\dot\gamma(t))-\rmd V(\PT{t}\,w)+\stackrel{R_t(s,w)}{\overbrace{\int_0^1(1-r) D^2F_t(rs,rw)((s,w), (s,w))dr}}.$$
The integral remainder satisfies the pointwise bound (\ref{error-1}).
Indeed, after shrinking the parameter neighbourhood, smoothness gives uniform bounds on $D^2F_t$ and $\partial_tD^2F_t$. Define the average
\begin{align*}
    \overline R_t
    :=|D_{\varepsilon,\eta}|^{-1}
    \int_{D_{\varepsilon,\eta}}
    R_t(s,w)\,\d(\Leb^1\otimes\vol_{h_0})(s,w).
\end{align*}
Averaging \eqref{error-1} gives
\begin{align*}
    |\overline R_t|+|\partial_t\overline R_t|
    \leq C(\eta^2+\varepsilon^2).
\end{align*}
The linear terms in the expansion of $F_t$ integrate to zero
by symmetry. Hence
\begin{align*}
    \meas\big[\mathscr{T}_{\gamma(t)}^{\varepsilon,\eta}\big]
    =\rme^{-V(\gamma(t))}
    |D_{\varepsilon,\eta}|(1+\overline R_t),
\end{align*}
which proves \eqref{density}. Moreover, the normalized density is
\begin{align*}
    \rho_t^{\varepsilon,\eta}(s,w)
    =\frac{F_t(s,w)}
    {|D_{\varepsilon,\eta}|(1+\overline R_t)}.
\end{align*}

Since $\gamma$ is a geodesic and $\PT{t}$ is parallel transport,
differentiating the expansion of $F_t$ yields
\begin{align*}
    \partial_tF_t(s,w)
    &=-s\Hess V(\dot\gamma(t),\dot\gamma(t))
    -\Hess V(\dot\gamma(t),\PT{t}\,w)
    +\partial_tR_t(s,w).
\end{align*}
Thus $|\partial_tF_t(s,w)|\leq C(\eta+\varepsilon)$
on $D_{\varepsilon,\eta}$.
For sufficiently small $\varepsilon,\eta$, we also have
$1+\overline R_t\geq\tfrac12$ and $|F_t|\leq C$.
Differentiating the normalized density therefore gives
\begin{align*}
    |D_{\varepsilon,\eta}|\,
    \partial_t\rho_t^{\varepsilon,\eta}(s,w)
    &=\frac{\partial_tF_t(s,w)}{1+\overline R_t}
    -\frac{F_t(s,w)\,\partial_t\overline R_t}
    {(1+\overline R_t)^2},
\end{align*}
and consequently $
    |D_{\varepsilon,\eta}|\,
    |\partial_t\rho_t^{\varepsilon,\eta}(s,w)|
    \leq C(\eta+\varepsilon)$.
Integrating between $u$ and $t$ proves the density comparison.
\end{proof}

\begin{lemma}[Moment lemma]\label{Le:TubeMoments}
For sufficiently small $\varepsilon,\eta>0$, the following expansions hold uniformly in $t\in I$:
\begin{align}\label{Eq:NormTube}
\begin{split}
\int_{D_{\varepsilon,\eta}} s\d\chi_{\gamma(t)}^{\varepsilon,\eta}(s,w) &= -\frac{\eta^2}{3}\,\rmd V(v) - \frac{t\eta^2}{3}\Hess V(v,v)\\
&\qquad\qquad + \rmO(\eta^3) + \rmO(\varepsilon^2\eta) + \rmO(t^2\eta^2),\\
a_{ij}\int_{D_{\varepsilon,\eta}} w^iw^j\d\chi_{\gamma(t)}^{\varepsilon,\eta}(s,w) &= \frac{\varepsilon^2}{n+2}\,\Ric(v,v) + \rmO(\varepsilon^4) + \rmO(\varepsilon^2\eta^2).
\end{split}
\end{align}
\end{lemma}

\begin{proof}
All remainders below are uniform in $t\in I$. For the first temporal moment, using the density formula  (\ref{density}) in \cref{Le:WeightedDensityExpansion}, symmetry,   the identity
$(2\eta)^{-1}\int_{-\eta}^{\eta}s^2\,\d s=\eta^2/3$,  and the error bound (\ref{error-1}), one has
\begin{align*}
    \int_{D_{\varepsilon,\eta}}
    s\,\d\chi_{\gamma(t)}^{\varepsilon,\eta}(s,w)
    &=\frac{
        -\frac{\eta^2}{3}\rmd V(\dot\gamma(t))
        +\rmO(\eta^3)+\rmO(\varepsilon^2\eta)
    }{
        1+\rmO(\eta^2)+\rmO(\varepsilon^2)
    }\\
    &=-\frac{\eta^2}{3}\rmd V(\dot\gamma(t))
    +\rmO(\eta^3)+\rmO(\varepsilon^2\eta).
\end{align*}
Since $\gamma$ is a geodesic, Taylor expansion at $t=0$ yields
\begin{align*}
    \rmd V(\dot\gamma(t))
    =\rmd V(v)+t\Hess V(v,v)+\rmO(t^2).
\end{align*}
Substitution proves the first identity in \eqref{Eq:NormTube}.

For the second spatial moments, fix $i,j\in\{1,\dots,n\}$.
After multiplication by $w^iw^j$, the linear terms in the
density expansion integrate to zero by symmetry.
The contribution of the remainder, after normalisation,
is $\rmO(\varepsilon^2\eta^2)+\rmO(\varepsilon^4)$.
Consequently,
\begin{align*}
    \int_{D_{\varepsilon,\eta}}
    w^iw^j\,\d\chi_{\gamma(t)}^{\varepsilon,\eta}(s,w)
    &=\vol_{h_0}\big[B_\varepsilon^{h_0}(0)\big]^{-1}
    \int_{B_\varepsilon^{h_0}(0)}
    w^iw^j\,\d\vol_{h_0}(w)\\
    &\quad+\rmO(\varepsilon^2\eta^2)+\rmO(\varepsilon^4).
\end{align*}
Multiplying by $a_{ij}$, summing over $i,j$, and applying
\cref{Co:RicciBall} proves the second identity in
\eqref{Eq:NormTube}.
\end{proof}

\providecommand{\TubeRevisionColor}{blue}

\subsection{Time separation expansion across tubes}\label{Sub:TimeSepTube}
The time separation between points in the two tubes admits the same
curvature correction as for slices, together with the temporal offset.
Under the scale assumptions below,
the term $\rmO(\eta\varepsilon^2)$ is absorbed into
$\rmO(\delta^2\varepsilon^2)$.
Below we work in a sufficiently small causally convex
neighbourhood contained in a convex normal neighbourhood covered
by this chart.   Recall that
$$ \Psi_{\gamma(t)}(s,w)
    :=\exp_{\gamma(t)}
    \bigl(s\dot\gamma(t)+\PT{t}\,w\bigr).$$

\begin{lemma}[Time separation between tube points]\label{Le:TubeTimeSep}
Fix $K,L>0$. Assume that $\delta,\varepsilon,\eta>0$
are sufficiently small, with
$\varepsilon=\rmo(\delta^{5/2})$ and $\eta\leq K\varepsilon$.
Uniformly for $s,r\in(-\eta,\eta)$ and
$w,w'\in B_\varepsilon^{h_0}(0)$ satisfying
\begin{align}\label{Eq:DisplBoundTube}
    |w-w'|_{h_0}
    =\rmO(\delta\varepsilon^2)+\rmO(\delta\eta^2),
\end{align}
we have
\begin{align*}
    &l\bigl(\Psi_x(s,w),\Psi_{\gamma(\delta)}(r,w')\bigr)\\
    &\quad=\delta+(r-s)-\frac{\delta}{2}a_{ij}w^iw^j
    +\rmO(\delta\varepsilon^3)+\rmO(\delta^2\varepsilon^2).
\end{align*}
Moreover, for all $s,r\in(-\eta,\eta)$ and $w,w'\in B_\varepsilon^{h_0}(0)$,  ( the displacement condition~\eqref{Eq:DisplBoundTube} is not needed),
 the following upper bound holds
\begin{align}\label{Eq:TubeUniversalUpper}
    &l\bigl(\Psi_x(s,w),\Psi_{\gamma(\delta)}(r,w')\bigr)\\
    &\quad\leq\delta+(r-s)-\frac{\delta}{2}a_{ij}w^iw^j
    +C\bigl(\delta\varepsilon^3+\delta^2\varepsilon^2\bigr).\notag
\end{align}
The constants are independent of
$\delta,\varepsilon,\eta,s,r,w,w'$, subject to the stated
assumptions. 

The constants are independent of the small parameters and the
endpoint choices; the subscripts indicate their permitted dependence
on $K$ and $L$.
\end{lemma}
\begin{proof}
Let $\Phi(t, w^1, \dots, w^n)=\exp_{\gamma(t)}(w^i e_i)$ be fixed Fermi chart and write
$p:=\Psi_x(s,w)$ and
$q:=\Psi_{\gamma(\delta)}(r,w')$, with  coordinates
$(t_-,z_-):=\Phi^{-1}(p)$ and $(t_+,z_+):=\Phi^{-1}(q)$.
For fixed $u$, set $a:=|s|+|w|_{h_0}$ and
\begin{align}\label{Eq:TubeFermiTransition}
   y_u(\tau):= \Phi^{-1}\bigl(\Psi_{\gamma(u)}(s,w)\bigr)
    =(u+s,w)+\rmO\bigl((|s|+|w|_{h_0})^3\bigr).
\end{align}
Then $y_u(0)=(u,0)$, $\dot y_u(0)=(s,w)$. Integrate twice, we have $y_u(1)=y_u(0)+\dot y(0)+\int_0^t (1-\tau) \, y{''}(\tau) d\tau$ and $y{''}(\tau)$. By the geodesic equation, using the fact that Christoffel symbols which vanish along $\gamma$,
smoothness gives $\Gamma(u,z)=\rmO(|z|_{h_0})$.
Since $\eta=\rmO(\varepsilon)$, this implies
$$
    (t_-,z_-)=(s,w)+\rmO(\varepsilon^3), \qquad 
    (t_+,z_+)=(\delta+r,w')+\rmO(\varepsilon^3).$$
Thus
$\zeta:=t_+-t_-=\delta+(r-s)+\rmO(\varepsilon^3)\sim\delta$.

For a spatial curve $z:[0,1]\to v^\perp$ joining $z_-$ to $z_+$,
write
$$
    c_z(\tau):=\Phi(t_-+\tau\zeta,z(\tau)).
$$
We first establish the bound
\begin{align}\label{Eq:TubeSpatialBounds}
    \|z\|_\infty+\|\dot z\|_\infty\leq C\varepsilon
\end{align}
for the two curves used below. The affine interpolation
$z_{\mathrm{aff}}(\tau):=(1-\tau)z_-+\tau z_+$
satisfies this bound, since
\begin{align*}
    \|z_{\mathrm{aff}}\|_\infty
    \leq\max\{|z_-|_{h_0},|z_+|_{h_0}\}
    \leq C\varepsilon, \qquad  \|\dot z_{\mathrm{aff}}\|_\infty
    =|z_+-z_-|_{h_0}
    \leq C\varepsilon.
\end{align*}
For any curve satisfying \eqref{Eq:TubeSpatialBounds},
the Fermi metric expansion gives
$$
    g(\dot c_z,\dot c_z)
    =\zeta^2-|\dot z|_{h_0}^2
    +\rmO\bigl((\zeta^2+|\dot z|_{h_0}^2)\varepsilon^2\bigr)
    =\zeta^2(1+\rmo(1))>0,
$$
because $\varepsilon=\rmo(\delta)$ and $\zeta\sim\delta$.
The affine interpolation therefore joins $p$ to $q$ by a
future-directed timelike curve.
Under our neighbourhood assumptions, the geodesic joining these
points is timelike and length-maximizing.
Parametrise it as
$c_z(\tau)=\Phi(t_-+\tau\zeta,z(\tau))$.
The Fermi metric expansion gives
$$
    g(\dot c_z,\dot c_z)
    =\zeta^2-|\dot z(\tau)|_{h_0}^2
    +\rmO\bigl((\zeta^2+|\dot z(\tau)|_{h_0}^2)\varepsilon^2\bigr)
    =\zeta^2(1+\rmo(1))>0,
$$
where we used $\varepsilon=\rmo(\delta)$ and $\zeta\sim\delta$.
 Hence $c_z$ is timelike and, since $\zeta>0$, future-directed.
The affine interpolation $z_{\mathrm {aff}}(\tau):=(1-\tau)z_-+\tau z_+$ satisfies
\begin{align*}
    \|z\|_\infty
    \leq\max\{|z_-|_{h_0},|z_+|_{h_0}\} \leq C\varepsilon, \quad 
    \|\dot z\|_\infty =|z_+-z_-|_{h_0}
    \leq C\varepsilon.
\end{align*}

For any curve satisfying \eqref{Eq:TubeSpatialBounds},
\cref{Le:LengthVariation} gives
\begin{align}\label{Eq:TubeCurveLength}
    \Len(c_z)
    &=\zeta-\frac{1}{2\zeta}
    \int_0^1|\dot z(\tau)|_{h_0}^2\,\d\tau\notag\\
    &\quad-\frac{\zeta}{2}\int_0^1
    a_{ij}(t(\tau))z^i(\tau)z^j(\tau)\,\d\tau
    +\rmO(\delta\varepsilon^3)
    +\rmO(\delta^2\varepsilon^2).
\end{align}
Indeed, with $\sigma:=\|\dot z\|_\infty=\rmO(\varepsilon)$
and $\zeta\sim\delta$, the original remainder satisfies
\begin{align*}
    &\rmO(\zeta\varepsilon^3)
    +\rmO(\varepsilon^2\sigma)
    +\rmO(\zeta^{-3}\sigma^4)=\rmO(\delta\varepsilon^3)
    +\rmO(\varepsilon^3)
    +\rmO(\delta^{-3}\varepsilon^4).
\end{align*}
The last two terms are absorbed into
$\rmO(\delta^2\varepsilon^2)$, since
$\varepsilon=\rmo(\delta^{5/2})$ implies
$$
    \varepsilon^3=\rmo(\delta^2\varepsilon^2),
    \qquad
    \delta^{-3}\varepsilon^4=\rmo(\delta^2\varepsilon^2).
$$
We also have $a_{ij}(t(\tau))=a_{ij}+\rmO(\delta)$,
uniformly in $\tau$.

For the lower bound, take $z=z_{\mathrm{aff}}$.
By \eqref{Eq:DisplBoundTube} and \eqref{Eq:TubeFermiTransition},
together with the scale assumptions,
\begin{align*}
    |\dot z|_{h_0}&=\rmO(\delta\varepsilon^2),&
    z(\tau)&=w+\rmO(\delta\varepsilon^2).
\end{align*}
The kinetic term in \eqref{Eq:TubeCurveLength} is therefore
$\rmO(\delta\varepsilon^4)$.
Substituting these estimates and the coefficient estimate above
into \eqref{Eq:TubeCurveLength}, and using
$\zeta=\delta+(r-s)+\rmO(\varepsilon^3)$, gives
\begin{align*}
    l(p,q)\geq\Len(c_z)
    &=\delta+(r-s)-\frac{\delta}{2}a_{ij}w^iw^j\\
    &\quad+\rmO(\delta\varepsilon^3)
    +\rmO(\delta^2\varepsilon^2).
\end{align*}

For the upper bound, use the maximizing geodesic considered above
and set
$
    E:=\int_0^1|\dot z(\tau)|_{h_0}^2\,\d\tau$.
By the fundamental theorem of calculus, Cauchy--Schwarz,
and $z_-=w+\rmO(\varepsilon^3)$,
$$
    |z(\tau)-w|_{h_0}\leq E^{1/2}+C\varepsilon^3.
$$
Together with the coefficient estimate above, this yields
\begin{align*}
    &\left|\int_0^1
    a_{ij}(t(\tau))z^i(\tau)z^j(\tau)\,\d\tau
    -a_{ij}w^iw^j\right|\\
    &\qquad\leq C\bigl(\delta\varepsilon^2
    +\varepsilon E^{1/2}+\varepsilon^4\bigr).
\end{align*}
Consequently, \eqref{Eq:TubeCurveLength} gives
\begin{align*}
    l(p,q)
    &\leq\zeta-\frac{\zeta}{2}a_{ij}w^iw^j
    -\frac{E}{2\zeta}+C\zeta\varepsilon E^{1/2}
    +C\bigl(\delta\varepsilon^3+\delta^2\varepsilon^2\bigr).
\end{align*}
By Young's inequality,
$$
    C\zeta\varepsilon E^{1/2}
    \leq\frac{E}{4\zeta}+C'\zeta^3\varepsilon^2,
$$
where $\zeta^3\varepsilon^2=\rmO(\delta^2\varepsilon^2)$.
Dropping the remaining negative kinetic term and converting
$\zeta$ to $\delta+(r-s)$ as in the lower bound proves
\eqref{Eq:TubeUniversalUpper}.
Together with the lower bound, this proves the asserted expansion.
\end{proof}

\subsection{Construction of a transport map}\label{Sub:TubeCorrection}
To obtain a coupling of the tube measures, it suffices to construct
a map $S\colon D_{\varepsilon,\eta}\to D_{\varepsilon,\eta}$
such that
\begin{align}\label{Eq:Spush}
    S_\push\chi_x^{\varepsilon,\eta}
    =\chi_{\gamma(\delta)}^{\varepsilon,\eta}.
\end{align}
Indeed, \eqref{Eq:muxpushforward} then gives the required map
\begin{align}\label{Eq:ThemapTinduces}
    T:=\Psi_{\gamma(\delta)}\circ S\circ\Psi_x^{-1}.
\end{align}
The calibration used in the main theorem will be specified in
\eqref{Eq:PrecCal}.

\begin{lemma}[Correction map]\label{Le:CorrectionMapTube}
Fix $c>0$ and let $\eta=c\varepsilon$.
For sufficiently small $\delta,\varepsilon>0$, there is a
diffeomorphism $S\colon D_{\varepsilon,\eta}\to D_{\varepsilon,\eta}$
satisfying \eqref{Eq:Spush},
extending to the closure with
$S(\partial D_{\varepsilon,\eta})=\partial D_{\varepsilon,\eta}$,
and
\begin{align}\label{Eq:ExpansionStube}
    S=\Id+\rmO(\delta\varepsilon^2).
\end{align}
\end{lemma}
\begin{proof}
Let $\rho_0$ and $\rho_1$ denote the densities of
$\chi_x^{\varepsilon,\eta}$ and
$\chi_{\gamma(\delta)}^{\varepsilon,\eta}$, respectively,
with respect to $\Leb^1\otimes\vol_{h_0}$, and set
$\rho_\theta:=(1-\theta)\rho_0+\theta\rho_1$ for
$\theta\in[0,1]$.
By \cref{Le:WeightedDensityExpansion} and
$\eta=\rmO(\varepsilon)$, uniformly on $D_{\varepsilon,\eta}$
and in $\theta\in[0,1]$,
\begin{align}\label{Eq:LLLL}
    c_0\leq |D_{\varepsilon,\eta}|\rho_\theta\leq C_0,
    \qquad
    |D_{\varepsilon,\eta}|\,|\rho_0-\rho_1|
    \leq C\delta\varepsilon.
\end{align}

Set $Q:=D_{1,1}$,
$L_\varepsilon(r,z):=(\eta r,\varepsilon z)$, and
$$a_\theta:=|D_{\varepsilon,\eta}|\,\rho_\theta\circ L_\varepsilon.$$
The Taylor expansion underlying
\cref{Le:WeightedDensityExpansion} is smooth in all variables.
Differentiating it in the rescaled variables therefore gives
\begin{align}\label{Eq:TubeScaledDensityDifference}
    \|a_0-a_1\|_{C^1(\overline Q)}\leq C\delta\varepsilon.
\end{align}
Here the factor $\delta$ comes from integrating the derivative
in the geodesic parameter, and every spatial derivative introduces
a factor $\eta$ or $\varepsilon$.
Moreover,
\begin{align}\label{Eq:UniEll}
    c_0\leq a_\theta\leq C_0
    \quad\text{on }Q,
\end{align}
and the $C^1$ norms of $a_\theta$ are uniformly bounded.

We construct a tangent vector field with prescribed divergence
on the fixed product domain.
Put $f:=a_0-a_1$ and
$\overline f(z):=\frac12\int_{-1}^1 f(r,z)\,\d r$.
Since $\int_Qf=0$, the Neumann problem
\begin{align*}
    \Delta_{h_0}\varphi=\overline f
    \quad\text{in }B_1^{h_0}(0),\qquad
    \partial_\nu\varphi=0
    \quad\text{on }\partial B_1^{h_0}(0)
\end{align*}
has a mean-zero solution. Define
\begin{align*}
    U_\parallel(r,z):=\int_{-1}^r
    \bigl[f(q,z)-\overline f(z)\bigr]\,\d q, \qquad     U_\perp(r,z):=\nabla^{h_0}\varphi(z).
\end{align*}
Then $\div_Q U=f$, and $U$ is tangent to every boundary face of
$Q$. Neumann estimates on the smooth ball and
\eqref{Eq:TubeScaledDensityDifference} yield
$\|U\|_{C^1(\overline Q)}\leq C\delta\varepsilon$.

Let $H_\theta$ be the flow of $X_\theta:=U/a_\theta$.
By construction,
\begin{align}\label{Eq:Lala}
    \partial_\theta a_\theta+\div_Q(a_\theta X_\theta)=0.
\end{align}
Consequently,
$(H_1)_\push(a_0\,\d r\,\d\vol_{h_0})
=a_1\,\d r\,\d\vol_{h_0}$.
The flow preserves every boundary face, and
$\|H_1-\Id\|_\infty\leq C\delta\varepsilon$.
Thus
$S:=L_\varepsilon\circ H_1\circ L_\varepsilon^{-1}$
has the required push-forward and boundary properties.
Since $\eta=c\varepsilon$, rescaling gives
$\|S-\Id\|_\infty\leq C\delta\varepsilon^2$.
\end{proof}

\subsection{Main result}\label{Sub:WeightedMain}
We now obtain the tube analogue of \cref{Th:MainInformal},
with the weight contributing the ambient Hessian of $V$.

\begin{theorem}[Reconstruction of damped timelike curvature]
\label{Th:WeightedReconstruction}
Let $(\mms,g)$ be a globally hyperbolic spacetime of dimension
$n+1$, where $n\in\N$, and let $V\in C^\infty(\mms)$.
Assume \cref{assumption1}.
For $\varepsilon>0$, define $\eta(\varepsilon)>0$ by
\begin{align}\label{Eq:PrecCal}
    \eta(\varepsilon)^2=\frac{3}{2(n+2)}\varepsilon^2.
\end{align}
For sufficiently small $\delta,\varepsilon>0$ with
$\varepsilon=\rmo(\delta^{5/2})$, the measures from
\eqref{Eq:mux} satisfy
\begin{align*}
    \ell_1\bigl(\mu_x^{\varepsilon,\eta(\varepsilon)},
    \mu_{\gamma(\delta)}^{\varepsilon,\eta(\varepsilon)}\bigr)
    &=\delta\Bigl[
    1-\frac{\varepsilon^2}{2(n+2)}
    \bigl(\Ric(v,v)+\Hess V(v,v)\bigr)\\
    &\qquad+\rmO(\varepsilon^3)+\rmO(\delta\varepsilon^2)
    \Bigr].
\end{align*}
In particular, the coarse timelike Ricci curvature from
\cref{Def:CTRC} associated with this family satisfies
\begin{align*}
    \lim_{\substack{\delta,\varepsilon\to0\\
    \varepsilon=\rmo(\delta^{5/2})}}
    \frac{2(n+2)}{\varepsilon^2}
    \kappa_\varepsilon(x,\gamma(\delta))
    =\Ric(v,v)+\Hess V(v,v).
\end{align*}
\end{theorem}

\begin{proof}
Write $\eta:=\eta(\varepsilon)$ and $D:=D_{\varepsilon,\eta}$,
and set
\begin{align*}
    m_t&:=\int_D s\,\d\chi_{\gamma(t)}^{\varepsilon,\eta}(s,w),&
    q_0&:=a_{ij}\int_D w^iw^j\,
    \d\chi_x^{\varepsilon,\eta}(s,w).
\end{align*}
Let $S=(S_\parallel,S_\perp)$ be given by
\cref{Le:CorrectionMapTube}, and let $T$ be defined by
\eqref{Eq:ThemapTinduces}.
The coupling $(\Id,T)_\push\mu_x^{\varepsilon,\eta}$
is chronological, and
$S_\perp-w=\rmO(\delta\varepsilon^2)$.
Integrating the expansion from \cref{Le:TubeTimeSep} along
the graph of $S$ and using its push-forward property gives
a lower bound with leading term
$\delta+m_\delta-m_0-\delta q_0/2$.
For the upper bound, pull an arbitrary admissible coupling
back to $D\times D$ and integrate \eqref{Eq:TubeUniversalUpper}.
The temporal offset is fixed by the two marginals, and the
curvature term by the source marginal, so both bounds give
\begin{align}\label{Eq:TubeMarginalReduction}
    \ell_1\bigl(\mu_x^{\varepsilon,\eta},
    \mu_{\gamma(\delta)}^{\varepsilon,\eta}\bigr)
    &=\delta+m_\delta-m_0-\frac{\delta}{2}q_0\\
    &\quad+\rmO(\delta\varepsilon^3)
    +\rmO(\delta^2\varepsilon^2).\notag
\end{align}

By \cref{Le:TubeMoments} and $\eta=\rmO(\varepsilon)$,
\begin{align}
    m_\delta-m_0
    &=-\frac{\delta\eta^2}{3}\Hess V(v,v)
    +\rmO(\varepsilon^3)+\rmO(\delta^2\varepsilon^2),
    \label{Eq:TubeTemporalMomentDifference}\\
    q_0
    &=\frac{\varepsilon^2}{n+2}\Ric(v,v)
    +\rmO(\varepsilon^4).\notag
\end{align}
The scale assumption gives
$\varepsilon^3=\rmo(\delta^2\varepsilon^2)$.
Thus the remainder in the temporal difference is absorbed into
$\rmO(\delta^2\varepsilon^2)$, while the spatial remainder
contributes only $\rmO(\delta\varepsilon^4)$ to
\eqref{Eq:TubeMarginalReduction}.
Substitution and the calibration
$\eta^2/3=\varepsilon^2/[2(n+2)]$ prove the asserted expansion.
By \cref{Def:CTRC}, this also gives
\begin{align*}
    \frac{2(n+2)}{\varepsilon^2}
    \kappa_\varepsilon(x,\gamma(\delta))
    =\Ric(v,v)+\Hess V(v,v)
    +\rmO(\varepsilon)+\rmO(\delta),
\end{align*}
and hence the curvature limit.
\end{proof}

\begin{remark}[On the calibration]\label{Re:Calibration}

Before caliberation, the temporal offset contributes
$-\delta\eta^2\Hess V(v,v)/3$, while the spatial term
contributes
$-\delta\varepsilon^2\Ric(v,v)/[2(n+2)]$.
The calibration \eqref{Eq:PrecCal} makes these coefficients equal. This also gives $\eta=c\varepsilon$ with a fixed positive constant, as required by the correction-map lemma.
\end{remark}

\section{Comparison with coarse extrinsic curvature}

A dimensional resemblance arises with the extrinsic formula
of Arnaudon--Li--Petko \cite{arnaudon-li-petko2025}
for an $m$-dimensional Riemannian submanifold
$N\subset\R^{m+k}$. Their test probability measures occupy
tubes of tangential radius $a$ and normal radius $\sigma$,
with $(m+k)$-dimensional support.
For $y=\exp^N_{x_0}(\delta e_1)$, they obtain
\begin{align}\label{extrinsic1}
    W_1(\mu_{x_0}^{\sigma,a},\mu_y^{\sigma,a})
    &=\|y-x_0\|
    \left[
        1+\left(
            \frac{\sigma^2}{k+2}
            -\frac{a^2}{2(m+2)}
        \right)
        \left\langle
            \mathrm{II}_{x_0}(e_1,e_1),H(x_0)
        \right\rangle
    \right]
    +\rmO(\delta^4).
\end{align}
Here $\mathrm{II}$ is the second fundamental form and
$H=\operatorname{tr}\mathrm{II}$ is the mean curvature
vector, without the normalizing factor $1/m$.
The formula holds for sufficiently small parameters satisfying
$\max\{\sigma,a\}\leq\delta/4$, provided that
$\mathrm{II}_{x_0}(e_1,e_j)=0$ for $j=2,\dots,m$,
where $(e_j)_{j=1}^m$ is an orthonormal basis of $T_{x_0}N$.
The comparisons below are formal: $W_1$ minimizes Euclidean
distance, whereas $\ell_1$ maximizes Lorentzian time separation.

Our unweighted construction uses measures supported on
$n$-dimensional spacelike slices. The identification
$$
    (m,k)=(n,1),\qquad
    a=\varepsilon,\qquad \sigma\downarrow0
$$
reduces the scalar coefficient in \eqref{extrinsic1} to
$-\varepsilon^2/[2(n+2)]$, agreeing with the prefactor
in Theorem~\ref{thm:unweighted}. The shared dimensional
factor comes from second moments of $n$-dimensional balls.
However, the centres in \eqref{extrinsic1} move along
the submanifold, whereas our centres move transversely
to the spacelike slices.

Retaining this hypersurface interpretation and introducing
normal thickness $\sigma=\eta$ also gives a comparison
of the calibrations. Our choice
$\eta^2=3\varepsilon^2/[2(n+2)]$ yields
$$
    \frac{\eta^2}{3}-\frac{\varepsilon^2}{2(n+2)}=0.
$$
Thus the same ratio of scales cancels the leading Euclidean
extrinsic contribution, even when the second fundamental
form is nonzero, while in our Lorentzian construction it
matches the coefficients of $\Ric(v,v)$ and $\Hess V(v,v)$.

Alternatively, organizing the tubes around the timelike
geodesic suggests
$$
    (m,k)=(1,n),\qquad
    a=\eta,\qquad \sigma=\varepsilon.
$$
This interpretation is closer to  \cite{arnaudon-li-petko2025} in the direction
of centre displacement: in both settings, the centres
move along a geodesic within the reference submanifold.
The Euclidean extrinsic coefficient is now
$$
    \frac{\varepsilon^2}{n+2}-\frac{\eta^2}{6}.
$$
Under our weighted calibration, it equals
$3\varepsilon^2/[4(n+2)]$, so the numerical agreement
of the scalar prefactors does not persist.

The geometric distinction is more substantial.
In ALP  \cite{arnaudon-li-petko2025}, the ambient space is flat, and for a curve the
extrinsic contraction in \eqref{extrinsic1} equals $|H|^2$.
Our reference curve is an ambient timelike geodesic, so
$$
    \mathrm{II}^{\gamma}(\dot\gamma,\dot\gamma)
    =\bigl(\nabla^g_{\dot\gamma}\dot\gamma\bigr)^\perp=0.
$$
The ambient spacetime, however, may be curved.
Transport between the surrounding tubes detects its
timelike Ricci curvature through the relative behaviour
of neighbouring geodesics and, with the calibrated weight,
recovers $\Ric+\Hess V$.

For $V=0$, under our scale assumptions, the leading Ricci
term is independent of the fixed positive ratio
$\eta/\varepsilon$. It therefore persists even at
$\eta^2=6\varepsilon^2/(n+2)$, the ratio that cancels
the leading ALP extrinsic contribution for a curve
with nonzero mean curvature.

\bigskip

{\bf Acknowledgments.} 
AI assistance was used to carry out routine computations, which was executed in May 2026, and typesetting and wording of the manuscript. After completion of the paper, the authors benefited from a discussion with Professor Georgios Moschidis that helped improve the presentation of the paper.

\bibliographystyle{amsrefs}
\bibliography{library}

\end{document}